\documentclass[12pt]{article}
\usepackage{amsmath,amssymb,amsthm,graphicx}
\usepackage[utf8]{inputenc}
\usepackage{wrapfig}
\usepackage{subcaption}
\usepackage{caption}
\usepackage[round]{natbib}
\usepackage[nottoc]{tocbibind}
\usepackage{url}
\usepackage[left=2.5cm,right=2.2cm,top=2.5cm,bottom=3.5cm]{geometry}
\usepackage{xcolor}
\usepackage{graphicx}
\usepackage[warn]{mathtext}
\usepackage{tikz}
\usepackage{pgfplots}
\usepackage{yfonts}
\usepackage[font=small,labelfont=bf]{caption}
\pgfplotsset{width=10.5cm,compat=1.9}
\usepackage{listings}
\usepackage{dsfont}
\usepackage{hyperref}
\usepackage{dsfont} 
\usepackage{breakcites}

\author{Egor Kravchenko}

\newcommand\one{\mathds 1}
\newcommand\N{\mathbb N}
\newcommand\R{\mathbb R}

\let\le\leqslant
\let\ge\geqslant

\def \P {\mathbb P}
\def \F {\mathcal F}
\newcommand \E {\mathbb E}
\newcommand \feasible {\mathfrak F}
\def \ve {\varepsilon}

\definecolor{ForestGreen}{rgb}{.13,.54,.13}
\definecolor{BrickRed}{rgb}{.80,.26,.33}
  \hypersetup{
           breaklinks=true,   
           colorlinks=true, 
           citecolor=blue,
           linkcolor=blue,
           pdfusetitle=true,  
        }

\newtheorem{theorem}{Theorem}[section]
\newtheorem{lemma}[theorem]{Lemma}
\newtheorem*{lemma*}{Lemma}
\newtheorem{proposition}[theorem]{Proposition}
\newtheorem{problem}[theorem]{Problem}

\newtheorem{corollary}[theorem]{Corollary} 	
\newtheorem*{problem*}{Задача}
\newtheoremstyle{Lemma_repeat}
        {\topsep}{\topsep}              
        {\itshape}                      
        {}                              
        {\bfseries}                     
        {.}                             
        { }                             
        {\thmname{#1}\thmnote{ \bfseries #3}}
    \theoremstyle{Lemma_repeat}

\newtheorem{Theorem_2}{Theorem}

\title{Coherent distributions: Hilbert space approach and duality
}
\begin{document}
\maketitle

\begin{abstract}

Let $X$ be a Bernoulli random variable with the success probability $p$. We are interested in tight bounds on 
$\mathbb{E}[f(X_1,X_2)]$, where  $X_i=\mathbb{E}[X| \mathcal{F}_i]$ and $\mathcal{F}_i$ are some sigma-algebras. This problem is closely related to understanding extreme points of the set of coherent distributions. A distribution on $[0,1]^2$ is called \emph{coherent} if it can be obtained as the joint distribution of $(X_1, X_2)$ for some choice of $\mathcal{F}_i$.
By treating random variables as vectors in a Hilbert space, we establish an upper bound for quadratic~$f$, characterize $f$ for which this bound is tight, and show that such $f$ result in exposed coherent distributions with arbitrarily large support. As a corollary, we get a tight bound on $\mathrm{cov}\,(X_1,X_2)$ for~$p\in [1/3,\,2/3]$.
To obtain a tight bound on $\mathrm{cov}\,(X_1,X_2)$ for all $p$, we develop an approach based on linear programming duality. Its generality is illustrated by tight bounds on $\mathbb{E}[|X_1-X_2|^\alpha]$ for any~$\alpha>0$ and $p=1/2$.
     
%
    
\end{abstract}

    \section{Introduction}

    A distribution $\mu$ on the unit square $[0, 1]^2$ 
    is called coherent if there exists a probability space~$(\Omega, \mathcal{F}, \mathbb{P})$, a Bernoulli random variable $X$ defined on this space, and two subalgebras~$\mathcal{F}_1, \mathcal{F}_2 \subset \mathcal{F}$ such that $\mu$ is a joint distribution of a random vector $(X_{1}, X_{2})$, \\ where~${X_{i} = \mathbb{E}[X|\mathcal{F}_i]}$. We denote the class of all these distributions by $\mathcal{C}$ and its subset corresponding to fixed expectation~$\mathbb{E}[X]=p$ by $\mathcal{C}_p$. 

    The set of coherent distributions $\mathcal{C}$ admits the following interpretation. Two experts make predictions about a random variable $X$, which they do not observe directly. The information available to expert $i$ is captured by a sigma-algebra $\mathcal{F}_i$. Accordingly, $X_{i}$ is expert $i$'s best prediction for $X$ given their information. Thus, the set $\mathcal{C}$ is the set of all possible distributions of experts' predictions, and the set $\mathcal{C}_p$ corresponds to experts attributing a prior probability $p$ to $\{X=1\}$.

    The key question about coherent distributions --- attracted substantial attention in probability and economic literature --- is how different experts' predictions can be.
    This question boils down to understanding $\mu\in \mathcal{C}_p$ maximizing $\int f(x_1,x_2)d\mu$, where $f$ is a certain measure of discrepancy. This problem turns out to be challenging, and the optimal $\mu$ is known for only a few particular objectives~$f$, e.g., 
     $f=(x_1-x_2)^2$ was analyzed by \citep*{cichomski2020maximal,arieli2020feasible} and $f=-(x_1-1/2)(x_2-1/2)$ corresponding to covariance minimization, by \citep*{arieli2020feasible} for~$p=1/2$.
In all the cases where the optimal $\mu$ is known, it has a simple form and contains at most~$4$ points in its support. 

Geometrically, maximization of $\int f(x_1,x_2)d\mu$ over  $\mathcal{C}_p$ corresponds to maximizing a linear functional over a convex set \citep{Dubins1980}. By Bauer's principle, the maximum is attained at an extreme point. Together with the simplicity of known solutions, this led to a conjecture that extreme points of $\mathcal{C}_p$ are atomic distributions with at most~$4$ points in the support \citep{burdzy2019bounds}.
This conjecture was refuted by \cite*{arieli2020feasible} and \cite{Zhu}, who constructed extreme $\mu$ with an arbitrarily large finite support; \cite{cichomski2023existence} demonstrated the existence of non-atomic extreme points.
Despite this progress in understanding extreme points of $\mathcal{C}_p$, no examples of objectives leading to non-elementary optima $\mu$ are known; in particular, it is not known whether the constructed extreme points are exposed.

The existence of exposed points with non-atomic support remains an open question. While our first theorem demonstrates that exposed $\mu$ with arbitrary large support originate as optima for seemingly innocent quadratic objectives~$f$. The optimal $\mu$ belongs to a family of ladder distributions depicted in Figure~\ref{Example of a ladder distribution}. We call $\mu$ a ladder distribution if its support can be represented as a sequence of points $((x_i, y_i))_{i=1}^n$, such that  

\begin{itemize}
    \item all points $(x_i, y_i)$ are distinct;
    \item sequences $(x_i)_{i=1}^n$ and $(y_i)_{i=1}^n$ are monotone (weakly increasing or decreasing);
    \item we have $x_i = x_{i+1}$ or $y_i = y_{i+1}$ for all $i \in \{1, \dots, n-1\}$;
    \item there are at most two points on any horizontal or vertical line.
\end{itemize}

    \begin{figure}[ht]
                \centering
                \begin{tikzpicture}[scale = 0.7]
                  \begin{axis}[
                    xlabel={$X_{1}$},
                    ylabel={$X_{2}$},
                    xmin=0,
                    xmax=1,
                    ymin=0,
                    ymax=1,
                    xtick={0, 0.2, 0.4, 0.6, 0.8, 1},
                    xmajorgrids=true,
                    ytick={0, 0.2, 0.4, 0.6, 0.8, 1},
                    ymajorgrids=true,
                    legend pos=south west,
                  ]
                  \addplot[only marks, color=black, ,mark=*,mark options={scale=2, fill=blue}] coordinates {(0.2, 1) (0.2, 0.8) (0.4, 0.8) (0.4, 0.6) (0.6, 0.6) (0.6, 0.4) (0.8, 0.4) (0.8, 0.2) (1, 0.2)};
                  \end{axis}
                \end{tikzpicture}%
           
            \caption{The support of a ladder distribution.}
            \label{Example of a ladder distribution}
    \end{figure}
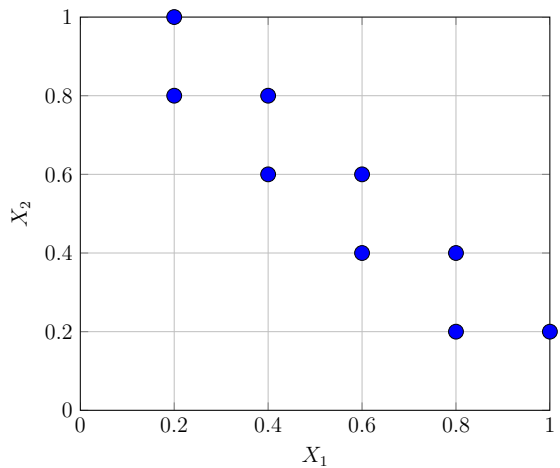

   \begin{theorem}
    \label{th1.1}
    Let 
    $$f(x_1,x_2)=\alpha(x_1-x_2)^2 + \beta(x_1+x_2-2p)^2.$$

    If  $\alpha < 0$ or $2\beta > \alpha$, then the following identity holds:

     $$\max_{\mu\in \mathcal{C}_p} \int f(x_1,x_2) d\mu = \max\left(0, 4\beta, \alpha + \beta \right)\cdot p(1-p).$$

    Otherwise, if $\alpha \ge \max(0, 2\beta)$, then the following bound holds:
    $$\max_{\mu\in \mathcal{C}_p} \int f(x_1,x_2) d\mu \leq p(1-p) \frac{\alpha^2}{\alpha - \beta}.$$

    
    Moreover, in this case, the equality is achieved if and only if there exists $N\in \mathbb{N}$ such that at least one of the following three conditions is satisfied:
    
    \begin{enumerate}
        \item  \label{case2}

         $ \frac {\alpha - \beta}{\alpha} = N$ and $ (2 - \frac{\alpha}{\alpha - \beta})p + \frac{\alpha}{\alpha - \beta} \ge 1 \ge (2 - \frac{\alpha}{\alpha - \beta})p .$
    
        \item \label{case1} 
        $\frac{2(\alpha - \beta) - (\alpha - 2\beta)p}{\alpha} = N$ and $(2 - \frac{\alpha}{\alpha - \beta})p + \frac{\alpha}{\alpha - \beta}> 1 $.
        
    
         \item  \label{case3}

         $\frac{2(\alpha - \beta) - (\alpha - 2\beta)(1-p)}{\alpha} = N$ and $(2 - \frac{\alpha}{\alpha - \beta})p < 1$.
    \end{enumerate}
In the case of equality, the optimum is attained at a ladder distribution with $2N-1$ points in the support if condition~\ref{case2} does not hold, and at a convex combination of two such distributions if it does hold.

\end{theorem}

Theorem~\ref{th1.1} relies on a Hilbert-space reinterpretation of the problem introduced by \citep*{arieli2020feasible, cichomski2020maximal}. The idea is to interpret $X$ as an element of $L_2(\Omega, \F, \P)$ and $X_{i}$ as its orthogonal projections. A quadratic $f$ can be expressed through a scalar product, and thus the value of the optimization is bounded by a maximum of some functional expressed through scalar products of projections of $X$. Not all orthogonal projections of $X$ correspond to conditional expectations, so this bound may not be tight.  Theorem~\ref{th1.1} 
exhaust all quadratic $f$ for which this approach provides the exact optimum.



Theorem~\ref{th1.1} gives us a tight lower bound on $\mathrm{cov}(X_{1}, X_{2})$ but not for all $p$. Indeed, for~$f(x_1, x_2) = \frac 14 (x_1-x_2)^2 +\frac 14(x_1+x_2-2p)^2=-(x_1-p)(x_2-p)$, maximizing $\int f(x_1, x_2) d\mu$ corresponds to minimizing $\mathrm{cov}(X_{1}, X_{2})$. The condition~\ref{case2} from Theorem~\ref{th1.1} simplifies to~${\frac 23 \ge p \ge \frac 13}$, so the bound is tight only for such those $p$.


To find a tight lower bound on $\mathrm{cov}(X_{1}, X_{2})$ for $p$ outside of $[1/3,\,2/3]$ we develop another approach, based on interpreting the problem of maximizing $\E[f(X_{1}, X_{2}]$ as a linear program. 
By the duality theorem, the value of the primal problem is bounded by the value of the dual.
Hence, the optimality of a candidate solution to the primal problem can be verified by constructing a candidate solution of the dual so that the values of the two objectives coincide.
We show how to construct a primal-dual pair of solutions and obtain the following theorem.

    \begin{theorem}
        Let $f(x_1, x_2) = -(x_1-p)(x_2-p)$. Then
        \label{cov_tight_bound}
        \begin{equation*}
            \max_{\mu \in \mathcal{C}_p} \int f(x_1, x_2) d\mu = \min_{(X_{1}, X_{2}) \in \mathcal{C}_p}(\mathrm{cov}(X_{1}, X_{2})) =  
            \begin{cases}
                -\left(\frac{p(1-p)}{1+p}\right)^2 & p\in \left(0, \frac 13\right)\\
                -\frac{1}8p(1-p) & \text{$p\in \left[\frac13, \frac23\right]$}\\
                -\left(\frac{p(1-p)}{2 - p}\right)^2 & p \in \left(\frac 23, 1\right).
            \end{cases}
        \end{equation*}
        For notational convenience, we write $(X_{1}, X_{2})\in \mathcal{C}_p$, whenever we want to indicate, that the distribution of the random vector $(X_{1},X_{2})$ belongs to $\mathcal{C}_p$.
    \end{theorem}


    Our linear programming duality method is not specific to the covariance or even quadratic objectives. We illustrate the generality by establishing a tight upper bound for $\E\{|X_{1} - X_{2}|^\alpha\}$. The bound for $\alpha \in (0, \alpha_0]$, where $\alpha_0 = 2.25751\dots$, was established by \cite*{arieli2023persuasion} as 
     $$\max_{(X_{1}, X_{2}) \in \mathcal{C}_{{1}/{2}}}\E[|X_{1} - X_{2}|^\alpha] = 2^{-\alpha}.$$
    
    While we get a bound for all $\alpha > 0$.
    
    \begin{theorem}
        \label{difference_tigh_bound}
      Let $f(x_1, x_2) = |x_1-x_2|^\alpha$ where $\alpha > 0$ and $p=1/2$. Then
       \begin{equation*}
          \max_{\mu  \in \mathcal{C}_\frac 12} \int f(x_1, x_2) d\mu  =
          \begin{cases}
            2^{-\alpha} & \text{ for } \alpha < \alpha_0\\
          
          2 \cdot \left(\frac{3\alpha + 1 - \sqrt{\alpha^2 +6\alpha + 1}}{2\alpha}\right)^\alpha \cdot  \frac{-(\alpha + 1) + \sqrt{\alpha^2 +6\alpha + 1}}{ \alpha - 1 + \sqrt{\alpha^2 +6\alpha + 1}} & \text{ for } \alpha \ge \alpha_0,
          \end{cases}
       \end{equation*}
       where $\alpha_0 = 2.25751\dots$ is the point, at which the function $ 2 \cdot \left(\frac{3\alpha + 1 - \sqrt{\alpha^2 +6\alpha + 1}}{2\alpha}\right)^\alpha \cdot  \frac{-(\alpha + 1) + \sqrt{\alpha^2 +6\alpha + 1}}{ \alpha - 1 + \sqrt{\alpha^2 +6\alpha + 1}}$ equals  $2^{-\alpha}$ .
    \end{theorem}
    
    The asymptotic behaviour of  $\max_{\mu \in  \mathcal{C}_{\frac 12}}\int |x_1 - x_2|^\alpha d\mu$ was studied  by~\cite{cichomski2023coherent}, Whose approach is more combinatorial.  For large  $\alpha$, the expression in  Theorem~\ref{difference_tigh_bound} simplifies (see Lemma~\ref{asymp} for details), and we obtain $\max_{\mu \in  \mathcal{C}_{\frac 12}}\int |x_1 - x_2|^\alpha d\mu \simeq \frac{1}{\alpha} \cdot \frac{2}{e} $ as~$\alpha\to \infty$, recovering the result of \cite{cichomski2023coherent}

    
    Besides, Theorem~\ref{difference_tigh_bound} provides a comprehensive answer to the open question formulated by \cite{burdzy2019bounds}, who asked what is a maximum of $\E[|X_{1} - X_{2}|^\alpha]$ for $({X_{1}, X_{2}) \in \mathcal{C}_{1/2}}$.

    \paragraph{Related literature.} 
    Coherent distributions were introduced  by~\cite{Dawid1995CoherentCO}. The study of coherent distributions was reinspired by \cite{burdzy2019bounds}. For instance, they found a tight bound for $\E[\one_{\{|X_{1} - X_{2}| \ge \delta\}}]$ and formulated multiple conjectures that inspired future studies of coherent distributions \citep*{Burdzy_Pal, 10.1214/21-EJP675,ARIELI2019185, arieli2023persuasion, he2023private}.


    From the probabilistic point of view, coherent distributions are closely related to copulas \citep{Zhu} and distributions with given marginals~\citep*{arieli2023persuasion, 10.1214/aoms/1177700153}. From the combinatorial perspective, coherent distributions arise in graph theory~\citep{cichomski2020maximal, cichomski2022combinatorial, tao2005szemeredis} and combinatorial matrix theory\citep{brualdi_2006, ryser_1957}.
    

    \section{Hilbert space approach and proof of Theorem~\ref{th1.1}}
       In this section, we show how to bound $\E\left[\alpha (X_{1} - X_{2})^2 + \beta(X_{1} + X_{2} - 2p)^2 \right]$ for coherent random vectors $(X_{1}, X_{2})$
       via a Hilbert space approach, study when the resulting bound is tight and obtain Theorem~\ref{th1.1} as a corollary.

        \bigskip

        The idea behind the Hilbert space approach originated in \citep*{cichomski2020maximal, arieli2020feasible} and is as follows.
        Let $(\Omega, \mathcal F, \P)$ be a probability space, and~$\mathcal H = \mathcal L_2(\Omega, \mathcal F, \P)$ be the corresponding Hilbert space with the inner product ${\langle Y, Z\rangle = \E[YZ]}.$ Let $X$ be a Bernoulli random variable with the success probability $p$, defined on a probability space $(\Omega, \mathcal F, \P)$, where $p$ is a fixed parameter. Suppose $\F_1, \F_2$ are sub sigma algebras of the sigma algebra $\F$, then we define random variables $X_{1}, X_{2}$ as~${X_{i} = \E[X| \F_i]}$.

        Initially, we center random variables $X, X_{1}, X_{2}$. Let ${\hat X = X - p}$ and ${\hat X_{i} = X_{i} - p = \E[\hat X| \mathcal F_i]}$.  Then we consider them as vectors in a Hilbert space $\mathcal H$. Random variables $\hat X_{i}$ are expectations of the random variable $X$, so, as vectors, $\hat X_{i}$ are projections of $\hat X$, which means that the angle between vectors~${\hat X_{i}}$ and~${\hat X - \hat X_{i}}$  equals $\frac {\pi}{2}$. So the end of the vector $\hat X_{i}$ lies on a sphere with diameter $\hat X$. This leads to the following observation.

            Let $\mathcal H = \mathcal L_2(\Omega, \mathcal F, \P)$ and $S \subset \mathcal H$ be a sphere with diameter $X$. Then 
            \begin{multline*}
                 \max_{(X_{1}, X_{2}) \in \mathcal C_p} \E\left[ \alpha (X_{1} - X_{2})^2 +  \beta(X_{1} + X_{2} -2p)^2 \right]  \le  \sup_{\hat X_{1}, \hat X_{2} \in S} \left( \alpha \|\hat X_{1} - \hat X_{2}\|^2 + \beta \|\hat X_{1} + \hat X_{2}\|^2 \right).
            \end{multline*} 

  
    The expression $\alpha \|\hat X_{1} - \hat X_{2}\|^2 + \beta \|\hat X_{1} + \hat X_{2}\|^2$ depends only on the mutual arrangement of vectors $\hat X, \hat X_{1}, \hat X_{2}$, so it is sufficient to consider only a sphere in $\R^3$ and the maximum instead of the supremum.

        \begin{lemma}[\cite*{cichomski2020maximal, arieli2020feasible}]
            \label{reduce_to_R}\
    
            Let~${S_2 \subset \R^3}$ be a sphere with diameter $\|X\|_{\mathcal H} = \sqrt{p(1-p)}$ and containing the point $0$. Then   
            $$ \max_{(X_{1}, X_{2}) \in \feasible} \E\left[ \alpha (X_{1} - X_{2})^2 + \beta(X_{1} + X_{2} -2p)^2 \right]  \le \max_{x_1, x_2 \in S_2} \left(\alpha \|x_1 - x_2\|^2 + \beta \|x_1 + x_2\|^2 \right).$$
        \end{lemma}

        We prove Lemma~\ref{reduce_to_R} in Appendix~\ref{appendix_reduce_to_R} for the reader's convenience. Our goal is to characterize when the bound from Lemma~\ref{reduce_to_R} is tight.

        First, we calculate the expression $\displaystyle \max_{x_1, x_2 \in S_2} \left(\alpha \|x_1 - x_2\|^2 + \beta \|x_1 + x_2\|^2 \right)$ via the standard Lagrange multipliers method, obtaining the following lemma proved in Appendix~\ref{Appendix_sphereoptlemma}:
        \begin{lemma}
            \label{sphereoptlemma}
            Let $S_2 \subset \R^3$ be the sphere built on a vector $w$ as a diameter. Then

            \begin{equation*}
                \max_{x_1, x_2 \in S_2} \left( \alpha \|x_1 - x_2\|^2 + \beta \| x_1 + x_2\|^2 \right) = \begin{cases}
                    \|w\|^2  \frac{\alpha^2}{\alpha - \beta} & \text{if } \alpha \ge \max(0, 2\beta) \\
                    \|w\|^2 \cdot \max\left(0, 4\beta, \alpha + \beta \right) & \text{otherwise}.
                \end{cases}
            \end{equation*}
            
            Moreover, in the case $\alpha > 0 \ge \beta$, the value $\|w\|^2 \frac{\alpha^2}{\alpha - \beta}$ is reached on those and only those pairs of vectors $x_1, x_2 \in S$ such that $x_1 + x_2 = \frac\alpha{\alpha - \beta}w$.
        \end{lemma}

        Combining this lemma with lemma \ref{reduce_to_R}, we obtain the following upper bound \\ on~$\E\left[\alpha (X_{1} - X_{2})^2 + \beta(X_{1} + X_{2} - 2p)^2 \right]$.

        \begin{corollary} The following identity holds
            \begin{equation*}
                \max_{(X_{1}, X_{2}) \in \mathcal{C}_p} \E\left[ \alpha (X_{1} - X_{2})^2 + \beta(X_{1} + X_{2} -2p)^2 \right]  \le  
                \begin{cases}
    				\frac{\alpha^2}{\alpha - \beta} \cdot p(1-p)  & \text{if } \alpha \ge \max(0, 2\beta) \\
    				\max\left(0, 4\beta, \alpha + \beta \right)\cdot p(1-p)  & \text{otherwise}.\\
			\end{cases}
            \end{equation*} 
            Moreover, in the case $\alpha \ge \max(0, 2\beta)  $, the value $\frac{\alpha^2}{\alpha - \beta} \cdot p(1-p)$ is reached on those and only those pairs of random variables $X_{1}, X_{2}$, such that $(X_{1} - p) + (X_{2} - p) = \frac\alpha{\alpha - \beta}(X - p)$. 
            \label{upper_bound_corollary}
        \end{corollary}

      This method provides us with an upper bound but does not say anything about whether this bound is tight or not. The values $0, 4\beta p(1-p), (\alpha + \beta)\cdot p(1-p)$ can be easily achieved:

    \begin{itemize}
        \item Let $X_{1} = X_{2} = \E[X] = p$. Then $\E\big[ \alpha (X_{1} - X_{2})^2 + \beta(X_{1} + X_{2} -2p)^2\big] = 0$.
        \item Let $X_{1} = X_{2} = X$. Then $\E\big[  \alpha (X_{1} - X_{2})^2 + \beta(X_{1} + X_{2} -2p)^2\big] = 4\beta p(1-p)$.
        \item Let $X_{1} = X, X_{2} = \E[X] = p$. Then $\E\big[ \alpha (X_{1} - X_{2})^2 + \beta(X_{1} + X_{2} -2p)^2 \big] = (\alpha + \beta)\cdot p(1-p)$.
    \end{itemize}

    Thus, this bound is tight if the condition $\alpha > \max(0, 2\beta)$ is not satisfied. The case $\alpha > \max(0, 2\beta)$ is more challenging. In this case, the bound is tight if and only if there exists a pair of random variables $(X_{1}, X_{2})  \in \mathcal{C}_p$ such that $(X_{1} - p) + (X_{2} - p) = \frac{\alpha}{\alpha - \beta}(X - p)$. Let's consider this situation more closely.
    
    \bigskip
    We start with a simple proposition.

    \begin{proposition}
        Let $A$ be a Borel subset of $\R$. If $A \subset (0, +\infty)$, then 
        $$ \P\Big(X_{1} \in A, X = 1\Big) = 0 \Longrightarrow  \P\Big(X_{1} \in A, X = 0\Big) = 0,$$
        and if $A \subset (-\infty, 1)$, then 
        $$ \P\Big(X_{1} \in A, X = 0\Big) = 0 \Longrightarrow \P\Big(X_{1} \in A, X = 1\Big) = 0.$$
        \label{prop}
    \end{proposition}

    \begin{proof}
        We will prove only the first statement; the proof of the second is identical and thus omitted.
        To begin, note that $\P(X_{1} \in A, X = 1) = \E[X\one_{X_{1} \in A}]$. The random variable $\one_{X_{1} \in A}$ is measurable with respect to the sigma algebra $\F_1$, so
        $$\P(X_{1} \in A, X = 1) = \E[X\cdot \one_{X_{1} \in A}] =\E\big[\E[X\cdot \one_{X_1 \in A} | \F_1]\big]   = \E\big[\E[X|\F_1] \one_{X_{1} \in A}\big] = \E[X_{1}\one_{X_{1} \in A}].$$

        If $\P(X_{1} \in A, X = 1) = 0$ then $ \E[X_{1}\one_{X_{1} \in A}] = 0$, so $\P(X_{1} \in A) = 0$. On the other hand,~$\P(X_{1} \in A, X = 0) \le \P(X_{1} = 0) = 0$.        
    \end{proof}

       Let $a$ be a positive real number and there exist coherent random variables $(X_{1}, X_{2})  \in \mathcal{C}_p$, such that $(X_{1} - p )+ (X_{2} - p) = a(X - p)$, then $X_{1} + X_{2} = aX + (2 - a)p$. This means that if $X = 0$, the distribution of $(X_{1}, X_{2})$ is concentrated on a line $X_{1} + X_{2} = (2 - a)p$ and if $X = 1$, the distribution of $(X_{1}, X_{2})$  is concentrated on a line $X_{1} + X_{2} = a + (2 - a)p$. The following proposition states that if we find set $I$, such that $\P(X_{1} \in I, X = 1) = 0$, then sets $I + ka$ for all integer $k \ge 0$ have the same property.

       \begin{proposition}
            \label{induction_prop}
           Let $I$ be a Borel subset of a half-infinite interval $\Big(\max\big(0, (2-a)p - 1\big), +\infty\Big)$, such that $$\P(X_{1} \in I, X = 1) = 0.$$ Then 
           $$\P\Big(X_{1} \in (I + ka), X = 1\Big) = 0$$ for all  integer $k \ge 0$.
       \end{proposition}

       \begin{proof}       
           We prove this proposition via induction on $k$. The base for $k = 0$ is given in the statement, so we only have to discuss the induction step. Let  $\P\big(X_{1} \in (I + ka), X = 1\big) = 0$, then by Proposition~\ref{prop}, we have

           \begin{gather*}
               0 = \P\Big(X_{1} \in (I + ka), X = 0\Big) = \P\Big(X_{2} \in \big((2-a)p - (I + ka)\big), X = 0\Big).
           \end{gather*}

            From the statement of the proposition we have $I > (2-a)p - 1$, so $(2-a)p - I <1$ and we can apply the Proposition~\ref{prop} again.

            $$0 = \P\Big(X_{2} \in \left((2-a)p - (I + ka)\right), X = 1\Big) = \P\Big(X_{1} \in \big(I+(k+1)a\big), X = 1\Big).$$
           
       \end{proof}

       Thus, we want to find a set $I$, such that $\P(X_{1} \in I, X = 1) = 0$. Generally, there are three cases of how lines  $X_{1} + X_{2} = a + (2 - a)p$ and $X_{1} + X_{2} = (2 - a)p$ are arranged.

    \begin{itemize} \item 
    The first case: both lines $X_{1} + X_{2}  = (2 - a)p$ and $X_{1} + X_{2} = a + (2 - a)p$ are above the line $X_{1} + X_{2} = 1$.  This condition is equivalent to $(2-a)p - 1 > 0$.

    We have $X_{2} \le 1$ almost surely, so 
    $$0 =  P\Big(X_{2} \in (1, a+1), X = 1\Big) = P\Big(X_{1} \in \big((2 - a)p - 1, (2-a)p - 1 + a\big), X = 1\Big). $$

    Then if we let $I_k = \Big((2 - a)p - 1, (2-a)p - 1 + a\Big)$,  by the Proposition~\ref{induction_prop}, we obtain~${P(X_{1} \in I_k, X = 1)}$ and,  by the Proposition~\ref{prop}, $\P(X_{1} \in I_k, X = 0) = 0.$

    Also, note that 
    $$0 = \P\big(X_{2} =1 + a, X = 1\big)  = \P\big(X_{1} =(2-a)p -1, X = 1\big)$$
    and, by the Proposition~\ref{prop}, $\P\big(X_{1} =(2-a)p -1, X = 0\big)= 0$.

    Thus we proved that distribution $\mu$ of $(X_{1}, X_{2})$  is concentrated on points of the \\ form~${\Big(ka+(2-a)p - 1, 1 - (k-1)a\Big)}$ for integer $k \ge 1$ and $\Big(ka+(2-a)p - 1, 1 - ka\Big)$ for integer $k \ge 1$. Consider $k = \lceil \frac{2 - (2-a)p}{a} \rceil$  and the point $\Big(ka+(2-a)p - 1, 1 - (k-1)a\Big)$, where $\lceil x \rceil$ denotes the  smallest integer greater than or equal to $x$. If $ka+(2-a)p - 1 > 1$ then 
    
    $$0 = \P\big(X_{1} = ka+(2-a)p - 1, X = 1\big) = \P\big(X_{2} = 1 - (k-1)a, X = 1\big),$$  so, by the Proposition~\ref{prop}, 
    $$0 = \P\big(X_{2} = 1 - (k-1)a, X = 0\big) = \P\big(X_{1} = (k-1)a + (2-a)p, X = 0\big).$$ 
    We have $(k-1)a + (2-a)p < 1$, so we can apply Proposition~\ref{prop} another time, and obtain $$0 = \P\big(X_{1} = (k-1)a + (2-a)p, X = 1\big).$$ Thus the distribution $\mu$ is zero at all points which leads to contradiction. This means that $ka+(2-a)p - 1 = 1$, that is~${k = \frac{2-(2-a)p}{a} \in \N}$.

    We found a support of the distribution $\mu$ of $(X_{1}, X_{2})$ and obtained what $\mu$ is a ladder distribution, now we find a distribution itself.

    Note that 
    \begin{gather*}
    \P(X_{i} = t, X = 1) = \E[X_{i}\one_{X_{i} = t}] = t\cdot \P(X_{i} = t), \\ 
    \P(X_{i} = t, X = 0) = \P(X_{i} = t) -  \P(X_{i} = t, X = 1) = (1 - t) \cdot \P(X_{i} = t).
     \end{gather*}

    So if $t > 0$, then 
    $$\P(X_{i} = t, X = 0) = \frac{1-t}{t} \cdot \P(X_{i} = t, X = 1),$$ and if $t < 1$, then $$\P(X_{i} = t, X = 1) = \frac{t}{1 - t} \cdot \P(X_{i} = t, X = 0).$$ Thus, if we fix $\mu\big((a+(2-a)p - 1, 1\big) = q$, we find $\mu$ at 
    all points. Then, using a condition that $\mu$ is a distribution, we find $q$. Thus in our case, support of the ladder distribution determines distribution in all points.

    An example of ladder distribution for $p = \frac 2{3}, a = \frac 1{5}$ is depicted in Figure~\ref{ladder_example_0}.

    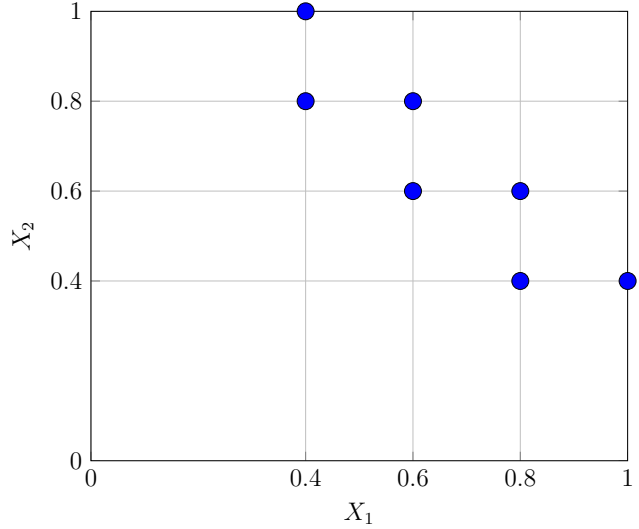
\begin{figure}[ht]
            \centering
            \begin{minipage}{0.4\linewidth}
           
                \centering
                \begin{tikzpicture}[scale = 0.8]
                  \begin{axis}[
                    xlabel={$X_{1}$},
                    ylabel={$X_{2}$},
                    xmin=0,
                    xmax=1,
                    ymin=0,
                    ymax=1,
                    xtick={0, 0.4, 0.6, 0.8, 1},
                    xmajorgrids=true,
                    ytick={0, 0.4, 0.6, 0.8, 1},
                    ymajorgrids=true,
                    legend pos=south west,
                  ]
                  \addplot[only marks, color=black, ,mark=*,mark options={scale=2, fill=blue}] coordinates {(0.4, 1) (0.4, 0.8) (0.6, 0.8) (0.6, 0.6) (0.8, 0.6) (0.8, 0.4) (1, 0.4)};
                  \end{axis}
                \end{tikzpicture}%
            \end{minipage}
            \caption{Support of $\mu$ for $p = \frac 23, a = \frac 15.$}
            \label{ladder_example_0}
    \end{figure}

    \item The second case: both lines are below the line $X_{1} + X_{2}$. This case boils down to the previous one. 
    To see this, we can consider $Y = 1 - X, Y_i = 1 - X_{i}$. Then $$Y_1 + Y_2 = 2 - X_{1} - X_{2} = 2 - aX - (2 - a)p = aY + (2 - a)(1 - p),$$ which leads us to the previous case and  the condition $\frac{2-(2-a)(1-p)}{a}\in \N$.

    \item The third case:    
    the line $X_{1} + X_{2} = a + (2 - a)p$ is above the line $X_{1} + X_{2} = 1$, and the line $X_{1} + X_{2} = (2 - a)p$ is below the line $X_{1} + X_{2} = 1$. 
    
    This case corresponds to $(2 - a)p - 1 < 0$ and $a + (2 - a)p > 1$.
    This means that
    $$0 = \P\Big(X_{1} \in \big((2 - a)p - 1\big)\Big) = \P\Big(X_{2} \in \big(1, a + (2 - a)p\big)\Big).$$

    Thus we proved that the distribution $\mu$ of $(X_{1}, X_{2})$ is zero on whole plane except two series of points: the first contains points of the form $\big((k+1)a + (2 - a)p - 1, 1-ka\big)$ for integer~$k \ge 0$, and the second contains points of the form $\big(ka, (2 - a)p-ka\big)$ for integer~$k \ge 0$. Each series ends when its point exits the square $[0, 1) \times (0, 1]$. Like in the first case, if this point is also out of the square $[0, 1] \times [0, 1]$, then the distribution of the vector $(X_{1}, X_{2})$ equals zero at this point, and, by Proposition~\ref{prop}, the distribution equals zero for all points in this series. Thus, there are three options for the distribution not to be zero on the whole plane:
    
    \begin{itemize}
        \item The last point of the first series lies on the line $X_{1} = 1$. Then its coordinates are~$\big(ka + (2 - a)p - 1, 1 - (k - 1)a\big)$, so $\frac{2 - (2 - a)p}{a} \in \N$. An example for $p = \frac 8{17}, a = \frac 3{10}$ is depicted in Figure~\ref{ladder1}.
        \item The last point of the second series lies on the line $X_{2} = 0$. Then its coordinates are~$\big(la, (2 - a)p - la\big)$, so $\frac{(2 - a)p}{a} \in \N$. Example for $p = \frac 6{11}, a = \frac 4{13}$ is depicted in Figure~\ref{ladder2}.
        \item The last point of the first series lies on the line $X_{2} = 0$. Then its coordinates are~$\big((k + 1)a + (2 - a)p - 1, 1 - ka\big)$, so $\frac{1}{a} \in \N$, and this is equivalent to the last point of the second series lying on the line $X_{1} = 1$.  Example for $p = \frac {1}{2}, a = \frac 13$ is depicted in Figure~\ref{ladder_convex_hull}.
    \end{itemize}
    
    As in the first case, if we know the total weight of the distribution $\mu$ on the ladder, then we know the weight of $\mu$ of each point of the ladder. So, in the first and the second situations, we know the whole distribution $\mu$, and in the third situation, the distribution $\mu$ is the convex hull of two ladder distributions corresponding to each series of points.

          \begin{figure}[h!]
            \centering
             \begin{minipage}{0.4\linewidth}
                \centering
                \begin{tikzpicture}[scale = 0.8]
                  \begin{axis}[
                    xlabel={$X_{1}$},
                    ylabel={$X_{2}$},
                    xmin=0,
                    xmax=1,
                    ymin=0,
                    ymax=1,
                    xtick={0,0.1, 0.4, 0.7, 1},
                    xmajorgrids=true,
                    ytick={0, 0.1, 0.4, 0.7, 1},
                    ymajorgrids=true,
                    legend pos=south west,
                  ]
                  \addplot[only marks, color=black, ,mark=*,mark options={scale=2, fill=blue}] coordinates {(0.1, 1) (0.1, 0.7) (0.4, 0.7) (0.4, 0.4) (0.7, 0.4) (0.7, 0.1) (1, 0.1)};
                  \end{axis}
                \end{tikzpicture}%
            \caption{Ladder for ${p = \frac 8{17}, a = \frac 3{10}}.$}
                 \label{ladder1}
            \end{minipage}
            \hfill
             \begin{minipage}{0.4\linewidth}
                \centering
                \begin{tikzpicture}[scale = 0.8]
                  \begin{axis}[
                    xlabel={$X_{1}$},
                    ylabel={$X_{2}$},
                    xmin=0,
                    xmax=1,
                    ymin=0,
                    ymax=1,
                    xtick={0, 4/13, 8/13, 12/13, 1},
                    xticklabels = {0, $\frac 4{13}$, $\frac8{13}$, $\frac {12}{13}$, 1},
                    xmajorgrids=true,
                    ytick={0, 4/13, 8/13, 12/13, 1},
                    yticklabels = {0, $\frac 4{13}$, $\frac8{13}$, $\frac {12}{13}$, 1},
                    ymajorgrids=true,
                    legend pos=south west,
                  ]
                  \addplot[only marks, color=black, ,mark=*,mark options={scale=2, fill=green}] coordinates {(0, 12/13) (4/13, 12/13) (4/13, 8/13) (8/13, 8/13) (8/13, 4/13) (12/13, 4/13) (12/13, 0)};
                  \end{axis}
                \end{tikzpicture}%
                \caption{Ladder for ${p = \frac 6 {11}, a = \frac {4}{13}}$.}
                \label{ladder2}
            \end{minipage}
        \end{figure}

        \begin{figure}[h!]
            \centering
             \begin{minipage}{0.4\linewidth}
                \centering
                \begin{tikzpicture}[scale = 0.8]
                  \begin{axis}[
                    xlabel={$X_{1}$},
                    ylabel={$X_{2}$},
                    xmin=0,
                    xmax=1,
                    ymin=0,
                    ymax=1,
                    xtick={0,1/6, 2/6, 3/6, 4/6, 5/6, 6/6},
                    xticklabels = {$0$,$\frac16$, $\frac26$, $\frac36$, $\frac46$, $\frac56$, $1$},
                    xmajorgrids=true,
                    ytick={0,1/6, 2/6, 3/6, 4/6, 5/6, 6/6},
                    yticklabels = {$0$,$\frac16$, $\frac26$, $\frac36$, $\frac46$, $\frac56$, $1$},
                    ymajorgrids=true,
                    legend pos=south west,
                  ]
                  \addplot[only marks, color=black, ,mark=*,mark options={scale=2, fill=blue}] coordinates {(5/6, 0) (5/6, 1/3) (5/6-1/3, 1/3) (5/6-1/3, 2/3) (5/6-2/3, 2/3) (5/6 - 2/3, 1)};
                  \addplot[only marks, color=black, ,mark=*,mark options={scale=2, fill=green}] coordinates {(0, 5/6) (1/3, 5/6) (1/3, 5/6-1/3) (2/3, 5/6-1/3) (2/3, 5/6-2/3) (1, 5/6 - 2/3) };
                  \end{axis}
                \end{tikzpicture}%
                \caption{Two ladders for ${p = \frac 12, a = \frac 13}$. The first one is coloured blue, the second one is colored green.}
                \label{ladder_convex_hull}
            \end{minipage}
        \end{figure}
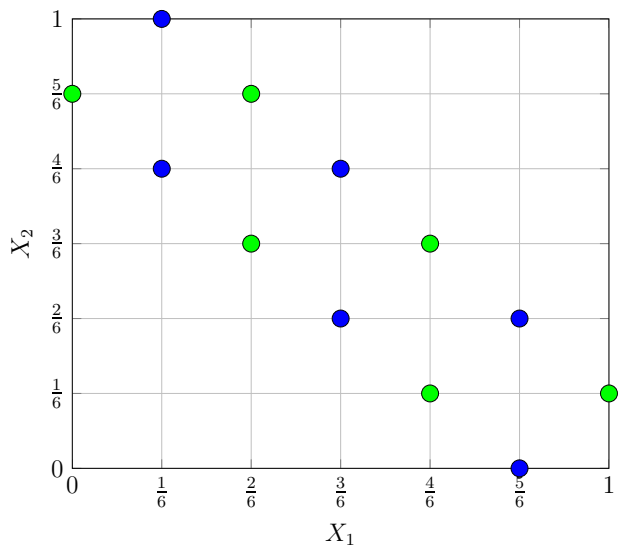        
    \end{itemize}

     Thus, by substituting $a = \frac{\alpha}{\alpha - \beta}$, we have identified all situations in which the upper bound stated in Corollary~\ref{upper_bound_corollary} is tight, leading us to the following theorem.

   \begin{theorem}
    \label{th1}
    Let 
    $$f(x_1,x_2)=\alpha(x_1-x_2)^2 + \beta(x_1+x_2-2p)^2.$$

    If  $\alpha < 0$ or $2\beta > \alpha$, then the following identity holds:

     $$\max_{\mu\in \mathcal{C}_p} \int f(x_1,x_2) d\mu = \max\left(0, 4\beta, \alpha + \beta \right)\cdot p(1-p).$$

    Otherwise, if $\alpha \ge \max(0, 2\beta)$, then the following bound holds:
    $$\max_{\mu\in \mathcal{C}_p} \int f(x_1,x_2) d\mu \leq p(1-p) \frac{\alpha^2}{\alpha - \beta}.$$

    
    Moreover, in this case, the equality is achieved if and only if there exists $N\in \mathbb{N}$ such that at least one of the following three conditions is satisfied:
    
    \begin{enumerate}
        \item  \label{case2'}

         $ \frac {\alpha - \beta}{\alpha} = N$ and $ (2 - \frac{\alpha}{\alpha - \beta})p + \frac{\alpha}{\alpha - \beta} \ge 1 \ge (2 - \frac{\alpha}{\alpha - \beta})p .$
    
        \item \label{case1'} 
        $\frac{2(\alpha - \beta) - (\alpha - 2\beta)p}{\alpha} = N$ and $(2 - \frac{\alpha}{\alpha - \beta})p + \frac{\alpha}{\alpha - \beta}> 1 $.
        
    
         \item  \label{case3'}

         $\frac{2(\alpha - \beta) - (\alpha - 2\beta)(1-p)}{\alpha} = N$ and $(2 - \frac{\alpha}{\alpha - \beta})p < 1$.
    \end{enumerate}
In the case of equality, the optimum is attained at a ladder distribution with $2N-1$ points in the support if condition~\ref{case2'} does not hold, and at a convex combination of two such distributions if it does hold.

\end{theorem}

    A lower bound for $\mathrm{cov}(X_{1}, X_{2})$ and $p=\frac{1}{2}$ was found  in \citep*{arieli2020feasible,cichomski2020maximal}. In both papers, authors rely on a version of the Hilbert space approach 
    to get an upper bound and then use an ad~hoc
    argument based on guessing the optimal distribution to establish tightness of the bound. The Theorem~\ref{th1.1} extends the result to $p \in [\frac{1}{3}, \frac{2}{3}]$ and provides the optimal distribution without guessing.
    
    \begin{corollary} Let $p$ be a real number in $(0, 1)$. If $p$ lies in the interval $[\frac 13,\frac 23]$, then
        $$\min_{(X_{1}, X_{2}) \in \mathcal{C}_p } \mathrm{cov}(X_{1}, X_{2}) = -\frac{1}{8} p(1-p).$$

        If $p$ lies outside the interval $[\frac 13,\frac 23]$, then the strict inequality holds:

         $$\min_{(X_{1}, X_{2}) \in \mathcal{C}_p } \mathrm{cov}(X_{1}, X_{2}) > -\frac{1}{8} p(1-p).$$
    \end{corollary}
    
    \begin{proof}
        Rewrite $\mathrm {cov}(X_{1}, X_{2})$ as a quadratic function of $X_{1}, X_{2}$
        $$\mathrm{cov}(X_{1}, X_{2}) = \E \big\{(X_{1} - p)(X_{2}-p)\big\} = -\frac{1}{4} \E\big\{(X_{1}-X_{2})^2-(X_{1}+X_{2}-2p)^2\big\}.$$
        Then we substitute $\alpha = 1, \beta = -1$ in Theorem~\ref{th1}. $\frac{\alpha- \beta}{\alpha} = 2 \in \N$, so we should check the condition $\frac{-\beta}{\alpha - \beta} \le  \frac{\alpha -2\beta}{\alpha - \beta}p  \le 1$. This means that $\frac{1}{2} \le \frac{3}{2} p \le 1$, which is $\frac{1}{3} \le p \le \frac{2}{3}.$
    \end{proof}

    \section{Linear programming approach}

    The Hilbert space approach is only applicable to quadratic functionals, yet even for these cases, it provides a tight bound only for a limited range of parameters. To find a tight bound in other situations, we need to use other techniques. This section discusses how linear programming duality helps us with our problem.
    

    We rely on the basic finite-dimensional duality. The following lemma shows that our infinite-dimensional optimization problem can be approximated with a finite-dimensional optimization problem by reducing to only finite subalgebras. An algebra is called finite if it contains only a finite number of sets. We prove the following lemma in Appendix~\ref{Appendix_discretization_lemma}.
    

    \begin{lemma}
    \label{discretization_lemma}
    Let $X$ be a bounded random variable defined on a probability space $(\Omega, \mathcal{F}, \P)$. For subalgebras $\F_1, \F_2$ of algebra $\F$, we define random variables $X_{i} = \E\left[X|\F_i\right]$ for $i = 1, 2$. 
    Let $\mathfrak{A}$ be a set of all subalgebras of $\F$ and $\mathfrak{B}$ be a set of all finite subalgebras of $\F$. 

    If a function $f:\mathbb{R}^2 \to \mathbb{R}^2$ is Lipschitz
    \footnote{For the sake of definiteness, we consider the regular Euclidean metric in $\mathbb{R}^2$, $\|(x, y)\| = \sqrt{x^2+y^2}$.}, then 
    $$\sup_{\F_1, \F_2 \in \mathfrak{A}}  \E f(X_{1}, X_{2}) = \sup_{\F_1, \F_2 \in \mathfrak{B}} \E f(X_{1}, X_{2}).$$
\end{lemma}

    This lemma says that it is enough to find the upper bound of $\E f(X_{1}, X_{2})$ only for atomic distributions of $(X_{1}, X_{2})$ with a finite number of atoms. Let 
    \begin{equation*}
        \begin{split}
            &\P\left(X_{1}, X_{2}\ | X = 1\right) = \sum_{1\le i, j \le n} \alpha_{i, j}   \delta _{\{X_{1} = a_i, X_{2} = a_j\}} \\
            &\P\left(X_{1}, X_{2}\ | X = 0 \right) =  \sum_{1 \le i, j \le n} \beta_{i, j}  \delta _{\{X_{1} = a_i, X_{2} = a_j\}},
        \end{split}
    \end{equation*}
    where $\delta$ is the Dirac measure. Without loss of generality, for every $i$, there is at least one nonzero value among $\alpha_{i, j}$ and $\beta_{i, j}$. Then, by Bayes' theorem, we have:
    \begin{equation*}
    	\begin{split}
    	   & a_i = \P\left(X_{1} = a_i|\ X=1\right) = \frac{\P\left(X_{1} = a_i, X=1\right)}{\P\left(X_{1}=a_i\right)} = \frac{p\sum_{j=1}^n \alpha_{i, j}}{p\sum_{j=1}^n \alpha_{i, j} + (1-p)\sum_{j=1}^n \beta_{i, j}},\\
    	   & a_j = \P\left(X_{2} = a_j|\ X=1\right) = \frac{\P\left( X_{2}= a_j, X=1\right)}{\P\left(X_{2}=a_j\right)} = \frac{p\sum_{i=1}^n \alpha_{i, j}}{p\sum_{i=1}^n \alpha_{i, j} + (1-p)\sum_{i=1}^n \beta_{i, j}}.
    	\end{split}
    \end{equation*}
    
    These equations can be 
    expressed as linear constraints on $\alpha_{i, j}$ and $\beta_{i, j}$.
    \begin{equation}
    	\label{first_constraint_of_primal_LP}
    	\begin{split}
    	   &  (1-a_i) p\sum_{j=1}^n \alpha_{i, j} - a_i (1-p)\sum_{j=1}^n \beta_{i, j} = 0,\\
    	   &  (1-a_j) p\sum_{i=1}^n \alpha_{i, j} - a_j (1-p)\sum_{i=1}^n \beta_{i, j} = 0.
    	\end{split}
    \end{equation}

    Note that constraints~\ref{first_constraint_of_primal_LP} are satisfied even if, for some $i$, all $\alpha_{i,j}$ and $\beta_{i, j}$ are zero for every~$j$.
   The only remaining constraints on $\alpha$ and $\beta$ are as follows:
    \begin{equation}
    	\label{second_constraint_of_primal_LP}
    	\begin{split}
    	   \alpha_{i, j} &\ge 0,\\
    	   \beta_{i, j} &\ge 0,\\
    	    \sum_{1 \le i, j \le n} \alpha_{i, j} &= 1,\\
    	   \sum_{1 \le i, j \le n} \beta_{i, j} &= 1.
    	\end{split}
    \end{equation}
    
    The objective function can also be rewritten as a linear function of $\alpha_{i, j}, \beta_{i, j}$:
    \begin{equation}
        \label{Objective_function_for_primal_LP}
    	\begin{split}
    	   \E[f(X_{1}, X_{2})] = \sum_{1 \le i, j \le n}f(a_i, a_j)(p\alpha_{i, j} + (1-p)\beta_{i, j}).
    	\end{split}
    \end{equation}

  Now, we can solve this linear 
  program numerically and find a distribution at which the objective function is very close to its optimal value. By the strong duality theorem, the optimal value of the objective function in the primal linear program equals the optimal value of the objective function in the dual problem. So, to prove that some value of the objective function in the primal linear program is optimal, it is enough to identify variables in the dual linear program such that the dual objective achieves identical value.

    Let us write down the dual problem. The $2n$ constraints of the form \eqref{first_constraint_of_primal_LP} turn into $2n$ variables~$\mu_i, \nu_i$. Constraints \eqref{first_constraint_of_primal_LP} are equations, so~$\mu_i, \nu_i \in \mathbb{R}$. Constraints of the form \eqref{second_constraint_of_primal_LP} become variables $\gamma, \delta \in \R$. Variables $\alpha_{i, j}, \beta_{i, j}$ turn into constraints.
    
    \begin{equation}
    	\label{Constraints_for_dual_LP}
    	\begin{split}
    	   \mu_i(1-a_i)p + \nu_j(1-a_j)p  + \gamma &\ge pf(a_i, a_j),\\
    	   -\mu_ia_i(1-p) - \nu_ia_j (1-p) + \delta  &\ge (1-p)f(a_i, a_j).
    	\end{split}
    \end{equation}
    
    And the objective function becomes $\gamma + \delta \to \min.$ 

    Note that conditions~\eqref{Constraints_for_dual_LP}  can be rewritten as:
    \begin{equation*}
        \begin{split}
            &\gamma \ge p \cdot \max_{1 \le i, j \le n} \Big\{f(a_i, a_j)-\mu_i(1-a_i) - \nu_j(1-a_j)\Big\},\\
            &\delta \ge (1-p)\max_{1 \le i, j \le n
            } \Big\{f(a_i, a_j)+\mu_ia_i + \nu_ja_j\Big\}.
        \end{split}
    \end{equation*}
    
     So the dual linear program takes the following form.

    \begin{problem}
        \label{discretr_problem_one}
        For a given number $p \in [0, 1]$, a function $f:[0, 1] \to \R$, and a sequence $(a_i)_{i=1}^n$, minimize the expression 
        \begin{equation*}
             p \cdot \max_{1 \le i, j \le n} \Big\{f(a_i, a_j)-\mu_i(1-a_i) - \nu_j(1-a_j)\Big\} + (1-p)\max_{1 \le i, j \le n
            } \Big\{f(a_i, a_j)+\mu_ia_i + \nu_ja_j\Big\}
        \end{equation*}
        over all sequences $(\mu_i)_{i=1}^n, (\nu_i)_{i=1}^n$ of real numbers.
    \end{problem}

    To find  $\displaystyle \sup_{(X_1, X_2) \in \mathcal C_p} \E f(X_1, X_2)$, we want to find the supremum among the minimum values. Consider the following problem:

      \begin{problem}
        \label{discretr_problem_two}
        For a given number $p \in [0, 1]$ and a function $f:[0, 1] \to \R$,  maximize the expression,
        \begin{equation*}
             \min_{(\mu_i)_{i=1}^n, (\nu_i)_{i=1}^n} \Bigg\{p \cdot \max_{1 \le i, j \le n} \Big\{f(a_i, a_j)-\mu_i(1-a_i) - \nu_j(1-a_j)\Big\} + (1-p)\max_{1 \le i, j \le n
            } \Big\{f(a_i, a_j)+\mu_ia_i + \nu_ja_j\Big\}\Bigg\}
        \end{equation*}
        over all numbers $n\in \N$ and sequences $(a_i)_{i=1}^n$ of numbers in $[0, 1]$. We call the supremum obtained from this maximization problem the value of the problem.
    \end{problem}
    
   Thus, combining strong duality theorem and Lemma~\ref{discretization_lemma} we obtain the following lemma:  

    \begin{lemma}
        The value of Problem~\ref{discretr_problem_two} equals $\displaystyle \sup_{(X_1, X_2) \in \mathcal C_p} \E f(X_1, X_2)$.
    \end{lemma}

    We will find $\mu_i$ and $\nu_i$ as functions of $a_i$, so  consider the continuous analogue of the problem~\ref{discretr_problem_two}:

    
    \begin{problem}
    	\label{reformulated_dual_LP}
    	For a given number $p \in [0, 1]$ and function $f:[0, 1] \to \R$, minimize the expression 
    	\begin{equation}
    	\label{reformulated_dual_LP_expression}
    	p\cdot \sup_{x, y \in [0, 1]} \Big\{f(x, y) - g(x)(1-x) - h(y)(1-y)\Big\}  + (1-p)\cdot \sup_{x, y \in [0, 1]} \Big\{f(x, y) + g(x)x + h(y)y\Big\}
    	\end{equation}
    
    	over all functions $g, h: [0, 1] \to \R$. We call the infimum obtained from this minimization problem the value of the problem.
    \end{problem}

    The following lemma states that the values of Problem~\ref{discretr_problem_two} and Problem~\ref{reformulated_dual_LP} coincide. We prove it in Appendix~\ref{Appendix_two_problems}.

    \begin{lemma}
        \label{two_problem_lemma} Let $f$ be a continuous function, then the value of Problem~\ref{discretr_problem_two} equals the value of Problem~\ref{reformulated_dual_LP}
    \end{lemma}

    As a result, the value of Problem~\ref{reformulated_dual_LP} equals to the $\displaystyle \sup_{(X_1, X_2) \in \mathcal C_p} \E f(X_1, X_2)$.



    Then we introduce two simplifications of this problem which are proven in Appendices~\ref{Appendix_dual_LP_first_simplification} and~\ref{Appendix_dual_LP_second_simplification}.

    \begin{lemma}
        \label{dual_LP_first_simplification}
        If $f$ in Problem~\ref{reformulated_dual_LP} is symmetric ($\forall x, y \in [0, 1]\ f(x, y) = f(y, x)$), then it is sufficient to minimize only over functions $g, h$, such that $g = h$.
    \end{lemma}
    
    \begin{lemma}
        \label{dual_LP_second_simplification}
       Let $p= \frac{1}{2}$, and let $f$ in Problem~\ref{reformulated_dual_LP} be symmetric, satisfying the condition ${f(x, y) = f(1 - x, 1- y)}$ for all ${x, y \in [0, 1]}$. Then, the value of Problem~\ref{reformulated_dual_LP} is equal to the infimum of the expression
        $$ \sup_{x, y \in [0, 1]} \Big\{f(x, y) + s(x)x + s(y)y\Big\}$$
        over all functions $s:[0, 1] \to \R$ such that $s(x) + s(1-x) = 0$.
    \end{lemma}
    
    \subsection{Tight bound on \texorpdfstring{$\mathrm{cov}(X_{1}, X_{2})$}{}}
    Theorem~\ref{th1} provides us a tight lower bound on $\mathrm{cov}(X_{1}, X_{2})$ for $p$ in $[\frac{1}{3}, \frac{2}{3}]$. This section discusses the following tight lower bound for other $p$.

    \begin{Theorem_2}[\ref{cov_tight_bound}]
    \begin{equation*}
    \min_{(X_{1}, X_{2}) \in \mathcal{C}_p}(\mathrm{cov}(X_{1}, X_{2})) =  
    \begin{cases}
    -\left(\frac{p(1-p)}{1+p}\right)^2 &p\in \left(0, \frac 13\right)\\
    -\frac{1}8p(1-p)&\text{$p\in \left[\frac13, \frac23\right]$}\\
    -\left(\frac{p(1-p)}{2 - p}\right)^2 	&p \in \left(\frac 23, 1\right).
    \end{cases}
    \end{equation*}
    \end{Theorem_2}

    \subsubsection{Optimal distribution}

    Assume $p \in (0, \frac 13]$. 
   First, we construct a distribution that minimizes $\mathrm{cov}(X_{1}, X_{2})$. It can be easily found numerically as a solution to the linear program with constraints~\eqref{first_constraint_of_primal_LP}, \eqref{second_constraint_of_primal_LP}, and objective function~\eqref{Objective_function_for_primal_LP} for $f(x, y) = -(x - p)(y - p)$. We selected $a_i = \frac{i}{n}$ for a large $n$ and solved this linear program for many different $p$. We observed that there exists one index~$i_0$, such that~${\alpha_{i_1, i_1} = 1, \beta_{0, i_1} = \beta_{i_1, 0} = \frac 12}$, and all other variables are approximately $0$. Thus we needed only to find the optimal value of $a_{i_1}$ that is routine. Thus we obtained the following distribution:
    
    \begin{equation*}
    \begin{split}
    \label{distribution_for_cov}
    \P\left(X_{1}, X_{2}\ | X =1\right) &= \delta _{\{X_{1} = \frac{2p}{p+1}, X_{2} = \frac{2p}{p+1}\}}, \\
    \P\left(X_{1}, X_{2}\ | X = 0 \right) &= \frac 12 \delta _{\{X_{1} = 0, X_{2} = \frac {2p}{p+1}\}} + \frac 12 \delta _{\{X_{1} = \frac{2p}{p+1}, X_{2} = 0\}}.
    \end{split}
    \end{equation*}
    
    It is easy to provide sigma algebras $\F_1, \F_2$ that generate this distribution: consider \\ sets~${\Theta = \{0, 1\}}$, $S_1 = S_2 = \{0, \frac{2p}{p+1}\}$, and the following distribution on the set $\Theta \times S_1 \times S_2$:

    \begin{equation*}
    \begin{split}
    &\P\left(\theta = 0, s_1 = 0, s_2 = \frac{2p}{p+1}\right) = \frac 12 (1-p),\\
    &\P\left(\theta = 0, s_1 = \frac{2p}{p+1}, s_2 = 0\right) = \frac 12 (1-p),\\
    &\P\left(\theta = 1, s_1 = \frac{2p}{p+1}, s_2 = \frac{2p}{p+1}\right) = p.
    \end{split}
    \end{equation*}

    Then let~$\F_i$ be a sigma-algebra of all subsets of $S_i$. The covariance can be calculated directly:
    
    \begin{gather*}
    \mathrm{cov}(X_{1}, X_{2}) = p\left(\frac{2p}{p+1} - p\right)^2 + (1-p)\left(\frac{2p}{p+1} - p\right)(0 - p) = \frac{p(p-p^2)}{p+1}\left(\frac{2p}{p+1} - 1)\right) = \\ 
    -\left(\frac{p(1-p)}{p+1}\right)^2.
    \end{gather*}

    \subsubsection{Lower bound}
        \label{cov_lower_bound_section}
       Further, we will prove that \( -\left(\frac{p(1-p)}{p+1}\right)^2\)  is a lower bound. As described above, we want to prove that the value of Problem~\ref{reformulated_dual_LP} is no more than $-\left(\frac{p(1-p)}{p+1}\right)^2$. The function $f(x, y) = -(x - p)(y - p)$ is symmetric, so, by Lemma~\ref{dual_LP_second_simplification}, we want to find a function $g$ such that
    \begin{equation*}
    \begin{split}
    &p\cdot \sup_{x, y \in [0, 1]} \{-(x-p)(y-p) -g(x)(1-x)-g(y)(1-y)\} + \\
    & (1-p)\cdot \sup_{x, y \in [0, 1]} \{-(x-p)(y-p) + g(x)x+g(y)y\} \le -\left(\frac{p(1-p)}{p+1}\right)^2.
    \end{split}
    \end{equation*}
    
    For every $p \in (0, \frac 13]$, this function can be found numerically as a solution to the following linear program:
    
    \begin{gather}
        \nonumber g(x), \gamma, \delta \in \R \\
         \label{cov_inequation1}\gamma \ge -(x - p)(y- p) - g(x)(1-x) - g(y)(1-y)\\
         \label{cov_inequation2}\delta \ge -(x-p)(y-p) +g(x)x+g(y)y\\
        \nonumber p\gamma + (1-p)\delta \to \min.
    \end{gather}

    We guess a solution $g$ and then verify our guess in Appendix~\ref{appendix_cov_lower_bound_section}.  The intuition is as follows. Let us assume that the first inequality turns into an equation for~${x \ge x_0, y = 0}$ for some $x_0$. Therefore $g(x) = \frac{\delta + p^2}{x} - p$ for $x \ge x_0$. Then, from the second inequality~$${\gamma \ge 2\delta + p^2+2p - 3(\delta + p^2)^\frac 23},$$ so we let $\gamma = 2\delta + p^2+2p - 3(\delta + p^2)^\frac 23$. The simplest way to define function $g(x)$ for $x \le x_0$ is to define it as a linear function. To satisfy \eqref{cov_inequation2}, it should be a function $g(x) = x_0 - p - \frac x2$. then for the function~$g$ to be continuous we let $x_0 = \sqrt{2(\delta + p^2)}$. Then the problem turns into the problem of minimizing $p\gamma + (1-p)\delta$ that is a function of one variable $\delta$. The minimum reaches for $\delta = \left(\frac{2p}{p+1}\right)^3 -p^2$. Eventually, $p\gamma + (1-p)\delta = -\left(\frac{p(1-p)}{p+1}\right)^2$. Proof that such function~$g$ satisfies the inequalities~\eqref{cov_inequation1},\eqref{cov_inequation2} is a technical routine.

    

    To obtain the lower bound for $p \in [\frac 23, 1)$, we notice that ${\mathrm{cov} (X_{1}, X_{2}) = \mathrm{cov} (1 - X_{1}, 1 - X_{2})}$, and $1 - X_{i}$ are conditional expectations of the random variable  $1- X$ with mean~${1 - p \in (0, \frac 13]}$. Therefore, for $p \in [\frac 23, 1)$, the lower bound on $\mathrm{cov} (X_{1}, X_{2})$ is  $-\left(\frac{p(1-p)}{2 - p}\right)^2$.

    \subsection{Tight bound on \texorpdfstring{$\E |X_{1} - X_{2}|^\alpha$}{}}

     The problem of finding a tight upper bound for $\E\{|X_{1} - X_{2}|^\alpha\}$ for $\alpha > 0$ was initially stated in \citep{burdzy2019bounds}, and it was solved there for $\alpha = 1$ using the bound for $\E \max(X_{1}, X_{2})$. Then, it was independently solved in \citep*{arieli2020feasible} and \citep{cichomski2020maximal} for $\alpha \in (0, 2]$ using the Hilbert space approach. Lately, it was solved in \citep*
     {arieli2023persuasion} for $\alpha \in (0, \alpha_0]$, where $\alpha_0 = 2.25751\dots$ using linear programming duality. In these cases, the upper  bound is~$2^{-\alpha}$, but, as stated in \citep*{cichomski2023coherent}, the asymptotic for $\alpha \to \infty$ differs a lot. This section establishes a tight upper  bound for all $\alpha \in [0, +\infty)$.

     \begin{Theorem_2}[\ref{difference_tigh_bound}] 
      Let $f(x_1, x_2) = |x_1-x_2|^\alpha$ where $\alpha > 0$ and $p=1/2$. Then
       \begin{equation*}
          \max_{\mu  \in \mathcal{C}_\frac 12} \int f(x_1, x_2) d\mu  =
          \begin{cases}
            2^{-\alpha} & \text{ for } \alpha < \alpha_0\\
          
          2 \cdot \left(\frac{3\alpha + 1 - \sqrt{\alpha^2 +6\alpha + 1}}{2\alpha}\right)^\alpha \cdot  \frac{-(\alpha + 1) + \sqrt{\alpha^2 +6\alpha + 1}}{ \alpha - 1 + \sqrt{\alpha^2 +6\alpha + 1}} & \text{ for } \alpha \ge \alpha_0,
          \end{cases}
       \end{equation*}
       where $\alpha_0 = 2.25751\dots$ is the point, at which the function $ 2 \cdot \left(\frac{3\alpha + 1 - \sqrt{\alpha^2 +6\alpha + 1}}{2\alpha}\right)^\alpha \cdot  \frac{-(\alpha + 1) + \sqrt{\alpha^2 +6\alpha + 1}}{ \alpha - 1 + \sqrt{\alpha^2 +6\alpha + 1}}$ equals  $2^{-\alpha}$ .
    \end{Theorem_2}
    
    \subsubsection{Optimal distribution}
    
    Fix $p = \frac 12$. There are two candidates for optimal distribution. 
    
    The first candidate is a trivial distribution $X_{1} = X, X_{2} = \frac 12$. In this case ${\E\big[|X_{1} - X_{2}|^\alpha\big] = 2^{-\alpha}}$. 
    This distribution is known to be optimal for $\alpha \le \alpha_0$.

    To guess the optimal distribution for big $\alpha$ we rely on numerical simulations. We guessed it via solving the primal linear program numerically for various $\alpha$ and found that the optimal distribution of $(X, X_{1}, X_{2})$ usually concentrates on six points $(1, a, 1), (1, 1, a), (1, 1 - a, 1 - a), (0, 0, 1-a), (0, 1 - a, 0), (0, a, a)$ for some $a$. Random variables $X_{1}, X_{2}$ are conditional expectations of~$X$, so, by Bayes' theorem, this distribution should be as follows 
    
    \begin{equation}
        \label{distribution_for_abs_power}
    	\begin{split}
    		  &\P\left(X = 1, X_{1} = a, X_{2} = 1\right) = \P\left(X = 1, X_{1} = 1, X_{2} = a\right)  = \frac{a}{2(a + 1)} \\
    		&\P\left(X = 1, X_{1} = 1 - a, X_{2} = 1 - a\right) = \frac{1- a}{2(a + 1)} \\
    		&\P\left(X = 0, X_{1} = 0, X_{2} = 1 - a\right) = \P\left(X = 0, X_{1} = 1 - a, X_{2} = 0\right) = \frac{a}{2(a + 1)} \\
    		&\P\left(X = 0, X_{1} = a, X_{2} = a\right) = \frac{1- a}{2(a + 1)}.
    	\end{split}
    \end{equation}
    
    It is easy to check that $X_{i}$ defined as in \eqref{distribution_for_abs_power} are conditional expectations of $X$ for all~${a \in [0, 1]}$. 
    For this distribution, $\E |X_{1} - X_{2}| =  (1 - a)^\alpha \cdot \frac{2a}{a + 1}$, so it is only left to maximize this expression over $a$. The optimal value is $a =  \frac{-(\alpha + 1) + \sqrt{\alpha^2 +6\alpha + 1}}{2\alpha}$, and for this $a$ we have 
    \begin{equation}
        \label{optimal_value_for_abs_power}
        \E\big[|X_{1} - X_{2}|^\alpha\big] = 2 \cdot \left(\frac{3\alpha + 1 - \sqrt{\alpha^2 +6\alpha + 1}}{2\alpha}\right)^\alpha \cdot  \frac{-(\alpha + 1) + \sqrt{\alpha^2 +6\alpha + 1}}{ \alpha - 1 + \sqrt{\alpha^2 +6\alpha + 1}}.
     \end{equation}
        Note that $ 2 \cdot \left(\frac{3\alpha + 1 - \sqrt{\alpha^2 +6\alpha + 1}}{2\alpha}\right)^\alpha \cdot  \frac{-(\alpha + 1) + \sqrt{\alpha^2 +6\alpha + 1}}{ \alpha - 1 + \sqrt{\alpha^2 +6\alpha + 1}} \ge  2^{-\alpha}$ exactly for $\alpha \ge \alpha_0$, so we claim that this distribution is optimal for big $\alpha \ge \alpha_0$.

    \subsubsection{Upper bound}

    We check that one of the two distributions described in the previous subsection gives the optimal value to $\E\big[|X_{1} - X_{2}|^\alpha\big]$, depending on the value of $\alpha$.

    \label{section_difference_lower_bound}

     Let $\mathrm{opt}(\alpha)$ be defined as follows: 
     \begin{equation*}
     \mathrm{opt}(\alpha) = 
       \begin{cases}
        2^{-\alpha} & \text{ for } \alpha < \alpha_0,\\
        2 \cdot \left(\frac{3\alpha + 1 - \sqrt{\alpha^2 +6\alpha + 1}}{2\alpha}\right)^\alpha \cdot  \frac{-(\alpha + 1) + \sqrt{\alpha^2 +6\alpha + 1}}{ \alpha - 1 + \sqrt{\alpha^2 +6\alpha + 1}} & \text{ for } \alpha \ge \alpha_0.
       \end{cases}
     \end{equation*} 
     By Lemma~\ref{dual_LP_second_simplification}, to prove  an upper bound, it is sufficient to find a function $s:[0, 1] \to \R$ such that $s(x) + s(1 - x) = 0$ and $\forall x, y \in [0, 1]$

    \begin{equation}
        \label{lower_bound_on_abs_power_eq1}
        |x-y|^\alpha+s(x)x+s(y)y \le \text{opt}(\alpha).
    \end{equation}
    
    We rely on the heuristic that inequality \eqref{lower_bound_on_abs_power_eq1} turns into equality for $x = 0, y \ge y_0$. This means that for $y \ge y_0$, $$y^\alpha + s(y)y = \text{opt}(\alpha),$$ so $$s(y) = \frac{\text{opt}(\alpha) - y^\alpha}{y}.$$ We are looking for a function~$s(x)$ under the condition $s(x) + s(1 - x) = 0$, so for $x \le 1 - y_0$, we have $$s(x) = \frac{(1-x)^\alpha -\text{opt}(\alpha)}{1 - x}.$$ Then for $y \ge y_0 \ge 1 - y_0 \ge x$ \eqref{lower_bound_on_abs_power_eq1} turns into
    
    \begin{equation}
        \label{lower_bound_on_abs_power_eq2}
        (y - x)^\alpha + x\frac{\text{opt}(\alpha) - (1 - x)^\alpha}{1 - x} - y^\alpha \le 0.
    \end{equation}
    
    For $x = 0$, the LHS equals $0$, so its derivative with respect to $x$ should be non-positive at the point $x = 0$, thus
    
    \begin{equation*}
        -\alpha y^{\alpha - 1} + 1 - \text{opt}(\alpha) \le 0.
    \end{equation*}
    
    The function $-\alpha y^{\alpha - 1} + 1 - \text{opt}(\alpha)$ monotonically decreases for $y$, so we obtain a lower bound 
    \[y_0 \ge \left(\frac{1 - \text{opt}(\alpha)}{\alpha}\right)^{\frac 1{\alpha - 1}}\] 
    and we let 
    \[y_0 = \max\left(\left(\frac{1 - \text{opt}(\alpha)}{\alpha}\right)^{\frac 1{\alpha - 1}}, \frac 12\right).\]
    
    For $x, y \le 1 - y_0$, we should check the inequality
    
    \begin{equation*}
        (y - x)^\alpha +x \cdot \frac{(1-x)^\alpha - \text{opt}(\alpha)}{1 - x} + y \cdot \frac{(1-y)^\alpha - \text{opt}(\alpha)}{1 - y}  \le \text{opt}(\alpha).
    \end{equation*}
    
    For $x = y$, it turns into 
    \begin{equation}
        \label{lower_bound_on_abs_power_eq3}
        \frac{2x \cdot (1-x)^{\alpha}} {1 + x} \le  \text{opt}(\alpha).
    \end{equation} 
    
    The function $\frac{2x \cdot (1-x)^{\alpha}} {1 + x}$ has a maximum for $x \in [0, 1]$ at the point $x_0$ such that 
    
    \begin{equation*}
       0 = \left(\frac{2x \cdot (1-x)^{\alpha}} {1 + x}\right)'\Bigg|_{x = x_0} = -\frac{2x_0(1-x_0)^{\alpha - 1}(\alpha x_0^2 +\alpha x_0 + x_0 - 1)}{(x+1)^2},
    \end{equation*}
    
    so $\alpha x_0^2 +(\alpha + 1) x_0  - 1 = 0$. Solving the quadratic equation, we get $x_0 = \frac{-(\alpha + 1) + \sqrt{\alpha^2 +6\alpha + 1}}{2\alpha}$. Then \eqref{lower_bound_on_abs_power_eq3} turns into
    
    \begin{equation*}
        2 \cdot \left(\frac{3\alpha + 1 - \sqrt{\alpha^2 +6\alpha + 1}}{2\alpha}\right)^\alpha \cdot  \frac{-(\alpha + 1) + \sqrt{\alpha^2 +6\alpha + 1}}{ \alpha - 1 + \sqrt{\alpha^2 +6\alpha + 1}} \le \text{opt}(\alpha),
    \end{equation*}

    The only difference between cases $\alpha < \alpha_0$ and $\alpha \ge \alpha_0$ is that for the case $\alpha \ge \alpha_0$ the inequality~\eqref{lower_bound_on_abs_power_eq3} turns into equality for $x = y = x_0$, while for $\alpha < \alpha_0$ it remains an inequality.
    
    All we have to do is extend the function $s(x)$ to $x \in (1-y_0, y_0)$, and we set it equal to $0$ there. Thus we obtain the function 
    
    \begin{equation*}
     s(x) = 
        \begin{cases}
            \frac{(1-x)^\alpha - \text{opt}(\alpha)}{1 - x} & x  \le \min( 1 - (1 - \text{opt}(\alpha))^{\alpha - 1}, \frac 12), \\
            \frac{\text{opt}(\alpha) - x^\alpha}{x} & x  \ge \max( (1 - \text{opt}(\alpha))^{\alpha - 1}, \frac 12), \\
            0 & \text{ otherwise.}
        \end{cases}
    \end{equation*}

    Proving that this function $s(x)$ satisfies the inequality \eqref{lower_bound_on_abs_power_eq1} is routine. We checked it via a Mathematica code, presented in Appendix~\ref{code}.

     \subsection{Asymptotic behavior of  \texorpdfstring{$\max_{\mu \in  \mathcal{C}_{\frac 12}}\int |x_1 - x_2|^\alpha d\mu$ as $\alpha \to \infty$}{}.}

    The asymptotic behavior of  $\max_{\mu \in  \mathcal{C}_{\frac 12}}\int |x_1 - x_2|^\alpha d\mu$ was found by \cite*{cichomski2023coherent} as $\frac 1\alpha \cdot \frac{2}{e}$ for large  $\alpha$. Theorem~\ref{difference_tigh_bound} allows us to check this result by direct computation.

    \begin{lemma}{\cite*{cichomski2023coherent}}
    \label{asymp}
        The following identity holds:
        \begin{equation*}
            \lim_{\alpha \to \infty} \alpha \cdot \Big(\max_{\mu \in  \mathcal{C}_{\frac 12}}\int |x_1 - x_2|^\alpha d\mu \Big) = \frac 2e.
        \end{equation*}

        \begin{gather*}
    \end{gather*}

    \end{lemma}

    \begin{proof}
        By Theorem~\ref{difference_tigh_bound}, for $\alpha > 2.25751\dots$ we have 
        \begin{equation}
        \label{limit1}
             \max_{\mu \in  \mathcal{C}_{\frac 12}}\int |x_1 - x_2|^\alpha d\mu = 2 \cdot \left(\frac{3\alpha + 1 - \sqrt{\alpha^2 +6\alpha + 1}}{2\alpha}\right)^\alpha \cdot  \frac{-(\alpha + 1) + \sqrt{\alpha^2 +6\alpha + 1}}{ \alpha - 1 + \sqrt{\alpha^2 +6\alpha + 1}}.
        \end{equation}

        To begin, we compute the limit of the second multiplier of the right side of~\eqref{limit1}.  Note that 
        \begin{equation*}
        \begin{split}
            &\frac{3\alpha + 1 - \sqrt{\alpha^2 +6\alpha + 1}}{2\alpha} = 1 + \frac{3\alpha + 1 - \sqrt{\alpha^2 +6\alpha + 1}}{2\alpha} = \\&
            1 + \frac{-4\alpha}{2\alpha \cdot (\alpha + 1 + \sqrt{\alpha^2 +6\alpha + 1})} = 
            1 - \frac{2}{\alpha + 1 + \sqrt{\alpha^2 +6\alpha + 1}}.
        \end{split}
        \end{equation*}
    Where  the expression $\frac{2}{\alpha + 1 + \sqrt{\alpha^2 +6\alpha + 1}}$ tends to $0$, as $\alpha$ tends to $\infty$. So 

    \begin{equation}
    \label{limit2}
        \begin{split}
           \lim_{\alpha \to \infty} \left(\frac{3\alpha + 1 - \sqrt{\alpha^2 +6\alpha + 1}}{2\alpha}\right)^\alpha &= \lim_{\alpha \to \infty} \left(1 - \frac{2}{\alpha + 1 + \sqrt{\alpha^2 +6\alpha + 1}}\right)^{\frac{\alpha + 1 + \sqrt{\alpha^2 +6\alpha + 1}}{2} \cdot \frac{2\alpha}{{\alpha + 1 + \sqrt{\alpha^2 +6\alpha + 1}}}} = \\ & e^{-\lim_{\alpha \to \infty} \frac{2\alpha}{{\alpha + 1 + \sqrt{\alpha^2 +6\alpha + 1}}}} = \frac 1e. 
        \end{split}
    \end{equation}

     Then we compute the limit of the last multiplier of the right side of~\eqref{limit1} multiplied by $\alpha$

     \begin{equation}
     \label{limit3}
          \lim_{\alpha \to \infty} \alpha \cdot \frac{-(\alpha + 1) + \sqrt{\alpha^2 +6\alpha + 1}}{ \alpha - 1 + \sqrt{\alpha^2 +6\alpha + 1}} = \lim_{\alpha \to \infty} \frac{\alpha \cdot 4\alpha}{(\alpha - 1 + \sqrt{\alpha^2 +6\alpha + 1)} \cdot (\alpha + 1 + \sqrt{\alpha^2 +6\alpha + 1})} = 1.
     \end{equation}

    Combining the equality~\eqref{limit1} with \eqref{limit2} and~\eqref{limit3} we obtain

    \begin{equation*}
        \lim_{\alpha \to \infty} \alpha \cdot \Big(\max_{\mu \in  \mathcal{C}_{\frac 12}}\int |x_1 - x_2|^\alpha d\mu \Big) = 2 \cdot \frac 1e \cdot 1 = \frac 2e.
    \end{equation*}
    
    \end{proof}

     \appendix
    \section*{Appendix}
    \addcontentsline{toc}{section}{Appendices}
    \renewcommand{\thesubsection}{\Alph{subsection}}

    \subsection{Proof of Lemma~\ref{reduce_to_R}}

    \label{appendix_reduce_to_R}
    
    \begin{proof}
        Let  $S \subset \mathcal H$ be a sphere with diameter $X$. 
        We established that 

         \begin{multline*}
                 \max_{(X_{1}, X_{2}) \in \mathcal C_p} \E\left[ \alpha (X_{1} - X_{2})^2 +  \beta(X_{1} + X_{2} -2p)^2 \right]  \le \\ \sup_{\hat X_{1}, \hat X_{2} \in S} \left( \alpha \|\hat X_{1} - \hat X_{2}\|^2 + \beta \|\hat X_{1} + \hat X_{2}\|^2 \right),
            \end{multline*}
        so we only need to prove that 
        \begin{equation*}
            \sup_{\hat X_{1}, \hat X_{2} \in S} \left( \alpha \|\hat X_{1} - \hat X_{2}\|^2 + \beta \|\hat X_{1} + \hat X_{2}\|^2 \right)  \le \max_{x_1, x_2 \in S_2} \left(\alpha \|x_1 - x_2\|^2 + \beta \| x_1 +  x_2\|^2 \right).
        \end{equation*}
        
        Consider a linear space $L = \mathrm {Lin}(\hat X, \hat X_{1}, \hat X_{2}) \subset \mathcal H$ and $S' = S \cap L$. $S'$ is a sphere of dimension no more than $3$, so there exists an isometry $\varphi: L \to \R^3$ that converts the sphere $S'$ to the sphere~$S_2$ or its subset. Let $x_1 = \varphi(\hat X_{1})$ and $x_2 = \varphi(\hat X_{2})$. We obtain that $$\alpha \|\hat X_{1} - \hat X_{2}\|^2 + \beta \|\hat X_{1} + \hat X_{2}\|^2 = \alpha\|x_1 - x_2\|^2 + \beta \|x_1 + x_2\|^2.$$ and so
        $$\sup_{\hat X_{1}, \hat X_{2} \in S} \left( \alpha \|\hat X_{1} - \hat X_{2}\|^2 + \beta \|\hat X_{1} + \hat X_{2}\|^2 \right)  \le \sup_{x_1, x_2 \in S_2} \left(\alpha \|x_1 - x_2\|^2 + \beta \| x_1 +  x_2\|^2 \right).$$
        
        It is only left to notice that $S_2$ is compact, and the function $\alpha \|x_1 - x_2\|^2 + \beta \| x_1 +  x_2\|^2 $ is continuous, so the maximum is achieved on some pair of points $x_1, x_2$.
    \end{proof}

    \subsection{Proof of Lemma~\ref{sphereoptlemma}}

    \label{Appendix_sphereoptlemma}
    
    \begin{proof}
        Let the maximal value of the left side be reached on vectors $x_1, x_2$. Firstly, we consider some partial cases.
        
        \begin{itemize}
            \item Case $x_1 = x_2$. Then  $\alpha \|x_1 - x_2\|^2 + \beta \|x_1 + x_2\|^2 = \beta \|2x_1\|^2 \le \|w\|^2 \max(0, \alpha + \beta ,4\beta)$.
            \item Case $x_1 + x_2 = w$. Then  $\alpha \|x_1 - x_2\|^2 + \beta \|x_1 + x_2\|^2 = \alpha \|x_1 - x_2\|^2 + \beta \|w\|^2 \le {\|w\|^2 \cdot \max( \alpha + \beta, \beta)} \le  \|w\|^2 \cdot \max(0, \alpha + \beta, \beta)$.
            \item Case $\alpha = 0$. Then $\alpha \|x_1 - x_2\|^2 + \beta \|x_1 + x_2\|^2  =  \beta \|x_1 + x_2\|^2  \le \|w\|^2 \max(0, \alpha + \beta, 4\beta)$.
            \item Case $\beta = 0$. Then  $\alpha \|x_1 - x_2\|^2 + \beta \|x_1 + x_2\|^2  =  \alpha \|x_1 - x_2\|^2  \le \|w\|^2 \max(0, \alpha + \beta, 4\beta)$.
        \end{itemize}
        
        So, the values $0, 4\beta\|w\|, (\alpha + \beta)\|w\|$ are achieved without any additional constraints. Let's consider the general case. The point $x$ lies on sphere $S_2$ if and only if $\|x-\frac{w}{2}\|  = \|\frac{w}{2}\|$. By the Lagrange multiplier theorem, there exist two real numbers $\lambda_1, \lambda_2$, such that
        \begin{gather*}
            \nabla_{x_1, x_2} \left(\alpha \|x_1 - x_2\|^2 + \beta \|x_1 + x_2\|^2 + \lambda_1 \left\|x_1 - \frac{w}{2} \right\|^2 + \lambda_2 \left\|x_2 - \frac{w}{2}\right\|^2\right) = 0.
        \end{gather*}
        
        Differentiating with respect to $x_1, x_2$, we obtain
        
        \begin{equation}
            \label{eq1}
            \begin{cases}
                \alpha (x_1 - x_2) + \beta (x_1 + x_2)  +\lambda_1 (x_1 - \frac{w}{2}) = 0 \\ 
                \alpha (x_2 - x_1) + \beta (x_1 + x_2)  +\lambda_2 (x_2 - \frac{w}{2}) = 0.
            \end{cases}
        \end{equation}
        
        Subtracting the first from the second, we have
        
        \begin{equation}
            \label{eq2}
            (2\alpha + \lambda_1)(x_1 - x_2) = (\lambda_1 - \lambda_2)\left(\frac{w}{2} - x_2\right).
        \end{equation}
        
        If both parts are not equal to $0$, then $(x_1 - x_2) \parallel \left(\frac{w}{2} - x_2\right)$, which means that ${\left(\frac{w}{2} - x_1\right) \parallel \left(\frac{w}{2} - x_2\right)}$, therefore endpoints of vectors $x_1, x_2, \frac{w}{2}$ lie on the same line. Then $x_1 + x_2 = w$ or~${x_1 = x_2}$. Both these cases we have already covered. Hence both parts of \eqref{eq2} equal $0$. This means that~$2\alpha + \lambda_1 = 0$. Analogously, $2\alpha + \lambda_2 = 0$. Then the equation \eqref{eq1} turns into 
        \begin{equation*}
            \alpha (x_1 - x_2) + \beta (x_1 + x_2)  - 2\alpha\left(x_1 - \frac{w}{2}\right) = 0.
        \end{equation*}
        Thus we have 
        $$(\beta - \alpha)(x_1 + x_2) = -\alpha w.$$ 
        Case $\alpha = 0$ we have already covered, so the right part does not equal $0$ and we can divide the equation by $\beta - \alpha$ and obtain
        
        \begin{equation}
            \label{eq3}
            x_1 + x_2 = \frac{\alpha}{\alpha - \beta}w.
        \end{equation}
        
        Such $x_1, x_2$ exist if and only if $0 \le \frac{\alpha}{\alpha - \beta} \le 2$.
        
        Now we calculate the value of  $\alpha \|x_1 - x_2\|^2 + \beta \|x_1 + x_2\|^2$. Vector $x_i$ lies on $S$, so ${\langle x_i - w, x_i \rangle = 0}$. Then multiplying \eqref{eq3} by $w$, we obtain
        $$\|x_1\|^2 + \|x_2\|^2= \langle x_1, w\rangle +\langle x_2, w \rangle =  \frac{\alpha}{\alpha - \beta}\|w\|^2.$$
        From the other side, squaring \eqref{eq3}, we obtain 
        $$\|x_1\|^2 + \|x_2\|^2 + 2\langle x_1, x_2 \rangle =  \left(\frac{\alpha}{\alpha - \beta}\right)^2 \|w\|^2.$$
        Using these two equalities we get 
        $$\|x_1 - x_2\|^2 = \left(2 \cdot \frac{\alpha}{\alpha - \beta} -  \left(\frac{\alpha}{\alpha - \beta}\right)^2 \right) \|w\|^2 =   \frac{\alpha \cdot (\alpha - 2\beta)}{(\alpha - \beta)^2}\cdot \|w\|^2.$$ 
        And finally, $$\alpha \|x_1 - x_2\|^2 + \beta \|x_1 + x_2\|^2 = \|w\|^2 \cdot \left(\frac{\alpha^2 \cdot (\alpha - 2\beta)}{(\alpha - \beta)^2}  +\beta \left(\frac{\alpha}{\alpha - \beta}\right)^2\right) =   \frac{\alpha^2}{\alpha - \beta}\cdot\|w\|^2 .$$
        
        To sum up, we find all possible extremums: $0, 4\beta\|w\|^2, (\alpha + \beta)\|w\|^2, \frac{\alpha^2}{\alpha - \beta}\cdot \|w\|^2$ and only we have to do is find all the situations, where the last value is greater than others. 

        \begin{itemize}
            \item For the value $0$ we have: $\frac{\alpha^2}{\alpha - \beta} \ge 0 \Longleftrightarrow \alpha \ge \beta$.
            
            \item If $\alpha > \beta$, for the value $(\alpha+ \beta)\|w\|^2$ we have: $\frac{\alpha^2}{\alpha - \beta} > \alpha + \beta \Longleftrightarrow \alpha^2 \ge \alpha ^2 - \beta ^2$.

            \item If $\alpha > \beta$
            for the value $4\beta \|w\|^2$ we have: $\frac{\alpha^2}{\alpha - \beta} \ge 4\beta \Longleftrightarrow (\alpha  - 2\beta)^2 \ge 0.$
        \end{itemize}

        Finally, conditions $0 \le \frac{\alpha}{\alpha - \beta} \le 2$ together with $\alpha  \ge \beta$ boil down to $\alpha \ge \max(0, 2\beta)$.
       
    \end{proof}

    \subsection{Proof of Lemma~\ref{discretization_lemma}}

    \label{Appendix_discretization_lemma}
    
    \begin{proof} 
        It is obvious that
        $$\sup_{\F_1, \F_2 \in \mathfrak{A}}  \E f(X_{1}, X_{2})   \ge \sup_{\F_1, \F_2 \in \mathfrak B} \E f(X_{1}, X_{2}),$$
        so we should only prove the converse inequality.
    
        Denote $$A = \sup_{\F_1, \F_2 \in \mathfrak A}   \E f(X_{1}, X_{2}) .$$ Fix $\ve > 0$, let $\delta = \frac{\ve}{\sqrt 2\lambda}$, and $\F_1, \F_2 \in \mathfrak A$, such that~$ \E f(X_{1}, X_{2}) \ge A - \ve$.
    
        Consider a sigma algebra $\hat \F_1$, generated by events $$\Bigg\{X_{1}\in \Big[k \delta , (k+1)  \delta \Big) , k \in \Big\{0, 1, \dots, \left\lceil \frac1{\delta} \right\rceil +1 \Big\}\Bigg\}.$$ 
        
        Analogously construct $\hat {\F_2}$. Note that $\hat \F_i \in \mathfrak B$.
    
        If  $X_{i} \in   \Big[k \delta , (k+1)  \delta \Big)$, then  $X_{\hat \F_i} \in  \Big[k \delta , (k+1)  \delta \Big)$, so $|X_{i} -   X_{\hat \F_i}| \le \delta$. The function $f$ is Lipschitz continuity function, so there exists constant $\lambda$, such that ${\Big|f(X_{1}, X_{2}) - f(X_{\hat \F_1}, X_{\hat \F_2}) \Big| \le \sqrt{2}\lambda \delta}$. Thus we obtain 
        \begin{equation*}
            \Big|\E f(X_{\hat \F_1}, X_{\hat \F_2}) - \E f(X_{1}, X_{2})\Big| \le \E \Big| f(X_{\hat \F_1}, X_{\hat \F_2}) - \E f(X_{1}, X_{2})\Big| \le \E \sqrt{2}\lambda \delta = \sqrt{2}\lambda \delta = \ve.
        \end{equation*}
        This means that  $ f(X_{\hat \F_1}, X_{\hat \F_2}) \ge A - 2\ve$. Then we tend $\ve \to 0$ and achieve
        $$\sup_{\F_1, \F_2 \in \mathfrak{A}}  \E f(X_{1}, X_{2})   \le \sup_{\F_1, \F_2 \in \mathfrak B} \E f(X_{1}, X_{2}).$$
    \end{proof}

    \subsection{Proof of Lemma~\ref{two_problem_lemma}}
    \label{Appendix_two_problems}
    \begin{proof}

    Let $V_1$ be the value of Problem~\ref{discretr_problem_two} and $V_2$ be the value of Problem~\ref{reformulated_dual_LP}.

    Consider  $g(x), h(x)$, such that  
        \begin{equation*}
             p\cdot \sup_{x, y \in [0, 1]} \Big\{f(x, y) - g(x)(1-x) - h(y)(1-y)\Big\}  + (1-p)\cdot \sup_{x, y \in [0, 1]} \Big\{f(x, y) + g(x)x + h(y)y\Big\} \le V_2 + \ve.
        \end{equation*} 
    \end{proof}

    Consider any sequence $(a_i)_{i=1}^n$ where $a_i \in [0, 1]$. Note that $\{a_1, a_2, \dots a_n\} \in [0, 1]$, so we can narrow the supremum to the set $\{a_1, a_2, \dots a_n\}$.
    \begin{equation*}
    \begin{split} 
    &p\cdot \sup_{x, y \in [0, 1]} \Big\{f(x, y) - g(x)(1-x) - h(y)(1-y)\Big\}  + (1-p)\cdot \sup_{x, y \in [0, 1]} \Big\{f(x, y) + g(x)x + h(y)y\Big\} \ge \\  &p\cdot \sup_{1 \le i, j \le n} \Big\{f(a_i, a_j) - g(a_i)(1-a_i) - h(a_j)(1-a_j)\Big\}  + (1-p)\cdot \sup_{1 \le i, j \le n} \Big\{f(a_i, a_j) + g(a_i)a_i + h(a_j)a_j\Big\}.
    \end{split}
     \end{equation*}

   Thus, $$V_2 + \ve \ge p\cdot \sup_{1 \le i, j \le n} \Big\{f(a_i, a_j) - g(a_i)(1-a_i) - h(a_j)(1-a_j)\Big\}  + (1-p)\cdot \sup_{1 \le i, j \le n} \Big\{f(a_i, a_j) + g(a_i)a_i + h(a_j)a_j\Big\},$$
   and we obtain that $V_2 + \ve \ge V_1$, so $V_2 \ge V_1$. Now we aim to prove the opposite inequality.

    The function $f$ is continuous on the square $[0, 1]^2$, so it is uniformly continuous. Fix $\ve > 0$, then there exists $\delta >0 $, such that $$\forall x_1, x_2, y_1, y_2 \in [0, 1]\ |x_1-x_2| + |y_1-y_2|< \delta \Rightarrow |f(x_1) - f(x_2)| < \ve.$$

   Fix a sequence $(a_i)_{i=1}^n$. Without loss of generality we may assume that there is a member of the sequence $(a_i)_{i=1}^n$ in every interval~${\left(x-\frac \delta4, x+\frac \delta4\right) \subset [0, 1]}$ and that $0 = a_1 < a_2 < \dots < a_n = 1$. 
    
    Consider the function $g(x)$ which equals $\mu_i$ at the interval $[a_i, a_{i+1})$ and equals $\mu_n$ at the point $1$. Analogously, construct the function $h(x)$ by sequence $\nu_i$. Note that if $a_i \le x < a_{i+1}$ then $|x-a_i| < |a_i-a_{i+1}| \le \frac \delta 2$. So, if  $a_i \le x < a_{i+1}$ and $a_j \le y < a_{j+1}$, then we have

    \begin{equation*}
    \begin{split}
   &V_2 \le p  \Big(f(x, y)-g(x)(1-x) - h(y)(1-y)\Big) + (1-p)\Big(f(x, y)+g(x)x + g(y)y\Big) \le \\
   &p\Big(f(a_i, a_j) + \ve -\mu_i(1-a_i) - \nu_j(1-a_j)\Big) + (1-p)\Big(f(a_i, a_j) + \ve +\mu_ia_i + \nu_ja_j\Big).
   \end{split}
   \end{equation*}

   So $V_2 \le V_1 + \ve$, and consequently $V_2 \le V_1$. Thus $V_2 = V_1$.

    \subsection{Proof of Lemma~\ref{dual_LP_first_simplification}}

    \label{Appendix_dual_LP_first_simplification}
    
    \begin{proof}
        Let 
        $$A = \sup_{x, y \in [0, 1]} \{f(x, y)-g(x)(1-x)-h(y)(1-y)\}.$$
        
        Consider a function $${t(x, y) = \frac{g(x, y) + h(x, y)}{2}}.$$

        For any $x, y \in [0, 1]$, we have 
        
        \begin{equation}
            \label{temp3}
        f(x, y)-g(x)(1-x)-h(y)(1-y) \le A,
        \end{equation}
        
        and 
        \begin{equation}
            \label{temp4}
        f(y, x)-g(y)(1-y)-h(x)(1-x) \le A.
        \end{equation}

        Summing up~\eqref{temp3}, \eqref{temp4} and dividing by two we have
        
         $$f(x, y) - t(x)(1-x) - t(y)(1-y) \le A.$$

        Thus, $$\sup_{x, y \in [0, 1]} \big\{f(x, y)-t(x)(1-x)-t(y)(1-y)\big\} \le A = \sup_{x, y \in [0, 1]} \big\{f(x, y)-g(x)(1-x)-h(y)(1-y)\big\}.$$ Analogously, $$\sup_{x, y \in [0, 1]} \{f(x, y) + t(x)x+t(y)y\} \le \sup_{x, y \in [0, 1]} \big\{f(x, y) + g(x)x+h(y)y\big\}.$$  
    \end{proof}

    \subsection{Proof of Lemma~\ref{dual_LP_second_simplification}}

    \label{Appendix_dual_LP_second_simplification}

    \begin{proof}
        Lemma ~\ref{dual_LP_first_simplification} states that the infimum over all functions $g, h: [0, 1] \to \R$ of 

        \begin{equation*}
             \frac 12 \cdot \sup_{x, y \in [0, 1]} \Big\{f(x, y)-g(x)(1-x)-h(y)(1-y) \Big\}  + \frac 12 \cdot \sup_{x, y \in [0, 1]} \Big\{f(x, y) + g(x)x+h(y)y \Big\}
        \end{equation*}

        equals the infimum over only function $g, h$ such that $g = h$. Let this infimum equals $A$.

        Note that the sum of two supremums is not less then the supremum of the sum, so 
        
        \begin{gather*}
             \frac 12 \cdot \sup_{x, y \in [0, 1]} \Big\{f(x, y)+h(x)x+h(y)y\Big\} +\frac 12 \cdot \sup_{x, y \in [0, 1]} \{f(x, y)-h(x)(1 - x)-h(y)(1 -y)\} \le \\ \sup_{x, y \in [0, 1]}\Big\{\frac 12 \cdot \Big(f(x, y) +  h(x)x + h(y)y\Big) + \frac 12  \cdot \Big(f(1 - x, 1 - y) -h(1 - x)x - h(1 - y)y\Big)\Big\}.
        \end{gather*}

        If we denote $s(x) = \frac 12 h(x) - \frac 12 h(1-x)$, then $s(x) + s(1-x) = 0$ and we have

        \begin{gather*}
             \sup_{x, y \in [0, 1]}\left\{p\Big(f(x, y) +  h(x)x + h(y)y\Big) + (1 - p)\Big(f(1 - x, 1 - y) -h(1 - x)x - h(1 - y)y\Big)\right\} = \\ \sup_{x, y \in [0, 1]} \Big\{f(x, y) +s(x)x + s(y)y \Big\}.
        \end{gather*}

        Thus, 

        $$A \le \inf_{s(x) + s(1-x)=0} \Bigg\{ \sup_{x, y \in [0, 1]} \Big\{f(x, y) +s(x)x + s(y)y \Big\}\Bigg\}.$$

        On the other hand, if function $S$ satisfies the condition $s(x)+s(1-x) = 0$, then

        \begin{gather*}
            A \le \frac 12 \cdot \sup_{x, y \in [0, 1]}  \Big\{f(x, y)-s(x)(1-x)-s(y)(1-y) \Big\}  + \frac 12 \cdot \sup_{x, y \in [0, 1]}  \Big\{f(x, y) + s(x)x+s(y)y \Big\} = \\
            \frac 12 \cdot \sup_{x, y \in [0, 1]}  \Big\{f(x, y)+s(1 - x)(1-x)+s(1 - y)(1-y) \Big\}  + \frac 12 \cdot \sup_{x, y \in [0, 1]}  \Big\{f(x, y) + s(x)x+s(y)y \Big\} = \\
            \sup_{x, y \in [0, 1]}  \Big\{f(x, y) + s(x)x+s(y)y \Big\}.
        \end{gather*}

        Thus,

        $$A = \inf_{s(x) + s(1-x)=0} \Bigg\{ \sup_{x, y \in [0, 1]} \Big\{f(x, y) +s(x)x + s(y)y \Big\}\Bigg\}.$$

    \end{proof}

    \subsection{Proof of lower bound in section~\ref{cov_lower_bound_section}}

        \label{appendix_cov_lower_bound_section}
        \begin{proof}
        
        We have the following designations:
        \begin{equation*}
        \begin{split}
            \delta &= \left(\frac{2p}{p+1} \right)^3 - p^2, \\
            \gamma &= 2\delta + p^2 + 2p - 3(\delta + p^2)^\frac 23,\\
            x_0 &= \sqrt{2(\delta + p^2)},\\
            g(x) &= \frac{\delta + p^2}x - p \text{ for $x \ge x_0$}, \\
            g(x) &= x_0 - p - \frac x2 \text{ for $x \le x_0$}.
        \end{split}
        \end{equation*}
        And we want to proof two following inequalities for $x, y \in [0, 1]$:

        \begin{gather}
             \label{cov_inequation3}\gamma \ge -(x - p)(y- p) - g(x)(1-x) - g(y)(1-y),\\
             \label{cov_inequation4}\delta \ge -(x-p)(y-p) +g(x)x+g(y)y.
        \end{gather}
        Note that function $g(x)$ is differentiable on $[0, 1]$. The function $-(x-p)(y-p)$ is also differentiable, so we have to check the inequalities~\eqref{cov_inequation3}, \eqref{cov_inequation4} only for $x, y \in \{0, 1\}$ and $x, y$ such that derivatives are zeros.

        First, we check \eqref{cov_inequation4}.

        \begin{itemize}
            \item For $x \ge x_0$, we have $\big(-(x-p)(y-p) + g(x)x\big)' = -y$. Then for $y = 0$  the inequality~\eqref{cov_inequation4} turns to 
            $$\delta \ge p(y-p) + \delta + p^2 - py = \delta.$$
            \item For $x \le x_0$, we have $\big(-(x-p)(y-p)+g(x)x\big)' = -(y-p) + (x_0-p) - x$, so $x + y = x_0$. Then the inequality~\eqref{cov_inequation4} turns to
            $$\delta \ge -xy -p^2 + x_0(x + y) -\frac {x^2}2 -\frac{y^2}2 = x_0^2 - \frac{x_0^2}2 -p^2 = \delta.$$
            \item For $x = 0, y = 0$ the inequality~\eqref{cov_inequation4} turns to $\delta \ge -p^2$.
            \item For $x = 0, y = 1$ the inequality~\eqref{cov_inequation4} turns to $\delta \ge p(1-p) + \delta + p^2 -p = \delta$.
            \item For $x = 1, y = 1$ the inequality~\eqref{cov_inequation4} turns to $\delta \ge -(1-p)^2 + 2\delta + 2p^2 -2p = \delta + \left(\frac{2p}{p+1} \right)^3 - 1$.
        \end{itemize}

        Then we check \eqref{cov_inequation3}.

        \begin{itemize}
            \item For $x \ge x_0$ we have $\big(-(x-p)(y-p) - g(x)(1-x)\big)' = \frac{\delta + p^2}{x^2} - y$.
            \item For $x \le x_0$ we have $\big(-(x-p)(y-p) - g(x)(1-x)\big)' = x_0 + \frac 12 - (x + y)$.
        \end{itemize}

        Thus if both partial derivatives are zeros, we have three cases:

        \begin{itemize}
            \item If $x \ge x_0, y \ge x_0$, then $x^2y = xy^2 = \delta + p^2$. So $x = y = \left(\delta + p^2\right)^\frac 13$ and the inequality~\eqref{cov_inequation3} turns to

            \begin{gather*}
                \gamma \ge 2p + p^2 +2\delta - xy - \frac{\delta + p^2}{x} - \frac{\delta + p^2}{y} = \gamma.
            \end{gather*}

            \item If $x \ge x_0, y \le x_0$ then we have $x + y = x_0 + \frac 12, x^2y = \delta + p^2$. Then the inequality~\eqref{cov_inequation3} turns into

            \begin{equation}
                \left(\frac{2p}{p+1}\right)^3 - 3\left(\frac{2p}{p+1}\right)^2 + \sqrt{2\left(\frac{2p}{p+1}\right)^3 }  \ge y\left(\frac y2 - x\right).
            \end{equation}

            If $p \le \frac 13$, then $x_0 \le \frac 12$. So $y \le x_0 \le \frac 12 \le x$. This means that $\frac y2 -x +xy \le 0$, so we can bound $\frac{y^2}2 -xy \le -x^2y =  -\left(\frac{2p}{p+1}\right)^3$. It is only left to note that $\frac{2p}{p+1} \le \frac 12$, so

            \begin{gather}
                \label{temp1}
                  2\left(\frac{2p}{p+1}\right)^3 - 3\left(\frac{2p}{p+1}\right)^2 + \sqrt{2\left(\frac{2p}{p+1}\right)^3 } \ge 2\left(\frac{2p}{p+1}\right)^3 - 3\left(\frac{2p}{p+1}\right)^2 + \frac{2p}{p+1} = \\
                  \frac{2p}{p+1} \left(\frac{2p}{p+1} - 1\right) \left(2 \cdot \frac{2p}{p+1} -1) \right) \ge 0.
            \end{gather}

            \item
            
            If $x, y < x_0$, then $x = y = \frac{x_0+\frac12}{2}$. But $x_0 \le \frac 12$, so $x = y \ge x_0$, which leads us to a contradiction.

            \item For $x = 1, y = 1$ the inequality~\eqref{cov_inequation3} turns to $\gamma \ge -(1-p)^2$, what follows from

            \begin{gather*}
                0\le 2\left(\frac{2p}{p+1}\right)^3 -3\left(\frac{2p}{p+1}\right)^2 + 1  = \left(\frac{2p}{p+1} - 1\right)^2\cdot \left(2\left(\frac{2p}{p+1}\right)+ 1\right).
            \end{gather*}

            \item For $x = 1, y = 0$ the inequality~\eqref{cov_inequation3} turns to $\gamma \ge p(1-p) - (x_0 -p)$. It can be rewritten as

            $$2\left(\frac{2p}{p+1}\right)^3 - 3\left(\frac{2p}{p+1}\right)^2 + \sqrt{2\left(\frac{2p}{p+1}\right)^3 } \ge 0,$$ what we have already proven in~\eqref{temp1}.

            \item For $x = 0, y = 0$ the inequality~\eqref{cov_inequation3} turns to $\gamma \ge -p^2 -2x_0 + 2p$, which follows from the previous case.
        \end{itemize}
           
    \end{proof}

    \subsection{Code for checking upper bound in section~\ref{section_difference_lower_bound}}

    \label{code}

     \begin{lstlisting}[language=Mathematica]
        ClearAll;
    f[x_, y_, alpha_] := Abs[x - y]^alpha; (*define function f*)
    opt[alpha_] := 
      2*((3*alpha + 1 - Sqrt[alpha^2 + 6*alpha + 1])/(2 alpha))^
        alpha*(-(alpha + 1) + Sqrt[alpha^2 + 6*alpha + 1])/(alpha - 1 + 
          Sqrt[alpha^2 + 6*alpha + 
            1]); (*define function opt for \alpha \ge \alpha_0*)
    roots = FindInstance[opt[alpha] - 2^(-alpha) == 0, {alpha}, Reals];
    alpha0 = alpha /. roots[[1]]; (*find alpha0*)
    y0[alpha_] := 
     Max[((1 - opt[alpha])/alpha)^(1/(alpha - 1)), 1/2] (*define y0*)
    s[alpha_, x_] := 
      Piecewise[{{((1 - x)^alpha - opt[alpha])/(1 - x), 
         x <= 1 - y0[alpha]}, {(opt[alpha] - x^alpha)/x, 
         x >= y0[alpha]}, {0, True}}]; (*define function s*)
    ineq[alpha_, x_, y_] := 
      f[x, y, alpha] + s[alpha, x]*x + s[alpha, y]*y - 
       opt[alpha]; (*set the inequation for \alpha \ge \alpha_0*)
    If[NMaximize[{ineq[alpha, x, y], 0 <= x <= 1, 0 <= y <= 1, 
         alpha >= alpha0}, {x, y, alpha}][[1]] <= $MachineEpsilon, 
     Print["Inequalities for \alpha \ge \alpha_0 are satisfied"], 
     Print["Inequalities for \alpha \ge \alpha_0 are not \
    satisfied"]](*check the inequality*)
    opt1[alpha_] = 2^(-alpha);
    ineq1[alpha_, x_, y_] := 
      f[x, y, alpha] + s[alpha, x]*x + s[alpha, y]*y - 
       opt1[alpha]; (*set the inequality for \alpha < \alpha_0*)
    If[NMaximize[{ineq1[alpha, x, y], 0 <= x <= 1, 0 <= y <= 1, 
         0 < alpha < alpha0}, {x, y, alpha}][[1]] <= $MachineEpsilon, 
     Print["Inequalities for \alpha < \alpha_0 are satisfied"], 
     Print["Inequalities for \alpha < \alpha_0 are not \
    satisfied"]](*check the inequality*)

    \end{lstlisting}

    \bibliography{sample}

\begin{thebibliography}{18}
\providecommand{\natexlab}[1]{#1}
\providecommand{\url}[1]{\texttt{#1}}
\expandafter\ifx\csname urlstyle\endcsname\relax
  \providecommand{\doi}[1]{doi: #1}\else
  \providecommand{\doi}{doi: \begingroup \urlstyle{rm}\Url}\fi

\bibitem[Arieli and Babichenko(2019)]{ARIELI2019185}
Itai Arieli and Yakov Babichenko.
\newblock Private bayesian persuasion.
\newblock \emph{Journal of Economic Theory}, 182:\penalty0 185--217, 2019.
\newblock ISSN 0022-0531.
\newblock \doi{https://doi.org/10.1016/j.jet.2019.04.008}.
\newblock URL
  \url{https://www.sciencedirect.com/science/article/pii/S0022053118302217}.

\bibitem[Arieli et~al.(2020)Arieli, Babichenko, Sandomirskiy, and
  Tamuz]{arieli2020feasible}
Itai Arieli, Yakov Babichenko, Fedor Sandomirskiy, and Omer Tamuz.
\newblock Feasible joint posterior beliefs, 2020.

\bibitem[Arieli et~al.(2023)Arieli, Babichenko, and
  Sandomirskiy]{arieli2023persuasion}
Itai Arieli, Yakov Babichenko, and Fedor Sandomirskiy.
\newblock Persuasion as transportation, 2023.

\bibitem[Brualdi(2006)]{brualdi_2006}
Richard~A. Brualdi.
\newblock \emph{Combinatorial Matrix Classes}.
\newblock Encyclopedia of Mathematics and its Applications. Cambridge
  University Press, 2006.
\newblock \doi{10.1017/CBO9780511721182}.

\bibitem[Burdzy and Pal(2021)]{Burdzy_Pal}
Krzysztof Burdzy and Soumik Pal.
\newblock Can coherent predictions be contradictory?
\newblock \emph{Advances in Applied Probability}, 53:\penalty0 133--161, 03
  2021.
\newblock \doi{10.1017/apr.2020.51}.

\bibitem[Burdzy and Pitman(2019)]{burdzy2019bounds}
Krzysztof Burdzy and Jim Pitman.
\newblock Bounds on the probability of radically different opinions, 2019.

\bibitem[Cichomski(2020)]{cichomski2020maximal}
Stanisław Cichomski.
\newblock Maximal spread of coherent distributions: a geometric and
  combinatorial perspective, 2020.

\bibitem[Cichomski and Os\c{e}kowski(2021)]{10.1214/21-EJP675}
Stanisław Cichomski and Adam Os\c{e}kowski.
\newblock {The maximal difference among expert’s opinions}.
\newblock \emph{Electronic Journal of Probability}, 26\penalty0
  (none):\penalty0 1 -- 17, 2021.
\newblock \doi{10.1214/21-EJP675}.
\newblock URL \url{https://doi.org/10.1214/21-EJP675}.

\bibitem[Cichomski and
  Os\c{e}kowski(2023{\natexlab{a}})]{cichomski2023coherent}
Stanisław Cichomski and Adam Os\c{e}kowski.
\newblock Coherent distributions on the square --extreme points and
  asymptotics, 2023{\natexlab{a}}.

\bibitem[Cichomski and
  Os\c{e}kowski(2023{\natexlab{b}})]{cichomski2023existence}
Stanisław Cichomski and Adam Os\c{e}kowski.
\newblock On the existence of extreme coherent distributions with no atoms,
  2023{\natexlab{b}}.

\bibitem[Cichomski and Petrov(2022)]{cichomski2022combinatorial}
Stanisław Cichomski and Fedor Petrov.
\newblock A combinatorial proof of the burdzy-pitman conjecture, 2022.

\bibitem[Dawid et~al.(1995)Dawid, DeGroot, Mortera, Cooke, French, Genest,
  Schervish, Lindley, McConway, and Winkler]{Dawid1995CoherentCO}
A.~P. Dawid, M.~H. DeGroot, Julia Mortera, R.~Cooke, Simon French, Christian
  Genest, M.~J. Schervish, Dennis~V. Lindley, K.~J. McConway, and Robert~L.
  Winkler.
\newblock Coherent combination of experts' opinions.
\newblock \emph{Test}, 4:\penalty0 263--313, 1995.

\bibitem[Dubins and Pitman(1980)]{Dubins1980}
Lester~E. Dubins and Jim Pitman.
\newblock A maximal inequality for skew fields.
\newblock \emph{Zeitschrift f{\"u}r Wahrscheinlichkeitstheorie und Verwandte
  Gebiete}, 52\penalty0 (3):\penalty0 219--227, Jan 1980.
\newblock ISSN 1432-2064.
\newblock \doi{10.1007/BF00538887}.
\newblock URL \url{https://doi.org/10.1007/BF00538887}.

\bibitem[He et~al.(2023)He, Sandomirskiy, and Tamuz]{he2023private}
Kevin He, Fedor Sandomirskiy, and Omer Tamuz.
\newblock Private private information, 2023.

\bibitem[Ryser(1957)]{ryser_1957}
H.~J. Ryser.
\newblock Combinatorial properties of matrices of zeros and ones.
\newblock \emph{Canadian Journal of Mathematics}, 9:\penalty0 371–377, 1957.
\newblock \doi{10.4153/CJM-1957-044-3}.

\bibitem[Strassen(1965)]{10.1214/aoms/1177700153}
V.~Strassen.
\newblock {The Existence of Probability Measures with Given Marginals}.
\newblock \emph{The Annals of Mathematical Statistics}, 36\penalty0
  (2):\penalty0 423 -- 439, 1965.
\newblock \doi{10.1214/aoms/1177700153}.
\newblock URL \url{https://doi.org/10.1214/aoms/1177700153}.

\bibitem[Tao(2005)]{tao2005szemeredis}
Terence Tao.
\newblock Szemer\'edi's regularity lemma revisited, 2005.

\bibitem[T.Zhu(2022)]{Zhu}
T.Zhu.
\newblock \emph{Some Problems on the Convex Geometry of Probability Measures}.
\newblock PhD thesis, UCBerkeley, 2022.
\newblock available at https://escholarship.org/uc/item/8001f519.

\end{thebibliography}
\end{document}